\documentclass{amsart}

\usepackage{amssymb}
\usepackage{amscd}
\usepackage{amsmath}
\usepackage{url}
\usepackage{xypic}

\begin{document}

\newtheorem{thm}[equation]{Theorem}
\newtheorem{prop}[equation]{Proposition}
\newtheorem{lemma}[equation]{Lemma}
\newtheorem{cor}[equation]{Corollary}
\newtheorem{conj}[equation]{Conjecture}
\theoremstyle{definition}
\newtheorem{definition}[equation]{Definition}
\newtheorem{remark}[equation]{Remark}
\newtheorem{example}[equation]{Example}
\numberwithin{equation}{section} 

\title[Special elements and Lichtenbaum-Gross]{On special elements \\
in higher algebraic $K$-theory \\
and the Lichtenbaum-Gross Conjecture}

\author{David Burns}
\address{King's College London\\Dept. of Mathematics\\London WC2R 2LS\\United Kingdom}

\author{Rob de Jeu}
\address{Faculteit der Exacte Wetenschappen\\Afdeling Wiskunde\\VU University Amsterdam\\De Boelelaan 1081a\\1081 HV Amsterdam\\The Netherlands}

\author{Herbert Gangl}
\address{Dept. of Math. Sciences\\University of Durham\\Durham DH1 3LE\\United Kingdom}

\begin{abstract}
We conjecture the existence of special elements in odd degree higher
algebraic $K$-groups of number fields that are related in a precise way to the values
at strictly negative integers of the derivatives of Artin
$L$-functions of finite dimensional complex representations. We prove this conjecture in certain important cases and also provide other
 evidence (both theoretical and numerical) in its support.
\end{abstract}



\maketitle

\def\Spec{\textup{Spec}}
\def\cok{\textup{cok}}

\def\C{\mathbb C}
\def\F{\mathbb F}
\def\P{\mathbb P}
\def\Q{\mathbb Q}
\def\R{\mathbb R}
\def\Z{\mathbb Z}

\def\CB{\mathcal B}
\def\CO{\mathcal O}

\def\emb#1{\Sigma_{#1}}
\def\embsp#1#2{\Sigma_{#2}^{#1}}
\def\Gal{\textup{Gal}}
\def\Hom{\textup{Hom}}
\def\I{(1+I)^\times}
\def\image{\textup{im}}
\def\ind{\textup{ind}}
\def\loc{\textup{loc}}
\def\NN{{\mathcal N}}
\def\ol{\overline}
\def\reg{\textup{reg}}
\def\Regr{\textup{Reg}_r}
\def\Tor{\textup{Tor}}

\def\chic#1#2{c_{#2}^{#1}}
\def\wr{w_r(\chi)}

\def\GL{GL_2(\mathbb{F}_3)}
\def\SL{SL_2(\mathbb{F}_3)}
\def\amG{\SL*\Z/4\Z}

\def\approx{\;{\buildrel . \over =}\;}

\section *{Introduction}

An important conjecture of Gross asserts, roughly speaking, that the leading term at any strictly negative integer of the Artin $L$-function of a finite dimensional complex representation should be equal, to within an undetermined algebraic factor, to a regulator constructed from elements of the appropriate odd degree higher algebraic $K$-group. (Gross's Conjecture was first formulated in the late 1970s as a natural analogue of the seminal conjecture of Stark concerning the leading terms at zero of Artin $L$-functions but was only recently published in \cite{gross} and can by now be seen as a special case of the natural equivariant refinement of Beilinson's general conjectures on the leading terms of $L$-functions.) It is well known that providing an explicit upper bound on the (absolute norm of the) denominator of the undetermined algebraic factor in Gross's Conjecture would make it much easier to obtain numerical evidence for the conjecture - see, for example, the discussion of Dummit in \cite[\S14]{dumm} regarding problems that arise when conducting numerical investigations of Stark's Conjecture. Perhaps more importantly, it is  also likely that such a bound could be used to give some much needed insight into arithmetic properties of any possible non-abelian analogues of the cyclotomic and elliptic elements in higher algebraic $K$-theory that have been constructed by Beilinson, Deligne and Soul\'e. However, apart from an important, but rather inexplicit, modification of the related Lichtenbaum-Gross Conjecture that was formulated by Chinburg, Kolster, Pappas and Snaith in \cite{ckps}, and a similarly inexplicit refinement of the Lichtenbaum-Gross Conjecture that we recently learnt has been formulated by Nickel in \cite{Nickel}, the present authors are not aware of any other predictions, let alone results, concerning this problem. 

In the current article we make a first step in this direction by investigating the possible existence of elements in odd degree higher
algebraic $K$-groups of number fields that are related in a very explicit way to the values at strictly negative integers of the first derivatives of Artin $L$-functions. By developing a suggestion of the first author in \cite[Rem. 5.1.5]{ltav}, which
was itself partly motivated by the results of Stark, Tate and Chinburg that are discussed by Tate in \cite[Chap. III]{tatebook}, we
formulate a precise conjecture in this regard (see Conjecture~\ref{chinlike} and Propositions ~\ref{extraprop}(i) and \ref{equivalences}). It is straightforward to see that this conjecture refines Gross's Conjecture and in Theorem \ref{mainres} we prove in addition that, in certain cases, it is implied by the modified Lichtenbaum-Gross Conjecture of \cite[Conj. 6.12]{ckps} and also that the elements it predicts should encode explicit information about the structure of even degree higher algebraic $K$-groups in a way that is strikingly
parallel to the way in which cyclotomic units are known to encode
information about the structure of ideal class groups (see Remark~\ref{rubin remark}). Via this
connection we are able to prove many cases of our conjecture for
characters which factor through Abelian extensions of $\Q$. For
 some cases of our conjecture that are not related to the Lichtenbaum-Gross Conjecture
we provide (in \S\ref{ne}) supporting numerical evidence for
various types of representations including certain dihedral and tetrahedral representations and also those studied by Tate and
Buhler \cite{buhler} and Chinburg \cite{chin}. As a preliminary step
to describing this evidence, which we believe may itself be of some
independent interest, in Theorem \ref{Bprop} we make precise the relation between a version
of the (second) Bloch group and $ K_3 $ of a field, and, if the
field is a number field, the Beilinson regulator map.

The first author would like to thank Dick Gross for important encouragement at an early stage of this project. It is also a pleasure to thank Andrew Booker and Xavier-Fran\c cois Roblot for their help with aspects of the numerical computations that are described in~\S\ref{ne}.

\section{Statement of the conjecture}

\subsection{}Throughout this article we use the following general notation.

We fix an algebraic closure $\Q^c$ of $\Q$ and for any Galois extension of fields $F/E$ we set $G_{F/E} := \Gal(F/E)$. For each non-zero integer $a$ and number field $F$ we write $\mu_a(F)$ for the finite module $H^0(G_{\Q^c/F},\Q/\Z(a))$ where
$\Q/\Z(a)$ denotes the group $\Q/\Z$ regarded as a $G_{\Q^c/\Q}$-module by setting $g(x) := \chi_{\rm cyc}(g)^ax$ for all $g \in G_{\Q^c/\Q}$ and $x \in \Q/\Z$ where $\chi_{\rm cyc}$ is the cyclotomic character. We also set
$w_a(F) := |\mu_a(F)|$ and $\R(a) := (2\pi i)^a\R$.

For any abelian group $A$ we write $A_{\rm tor}$ for its torsion subgroup. Unadorned tensor products are to be regarded as taken in the category of abelian groups.

\subsection{}\label{soc sect}
To state our conjecture we fix a finite Galois extension of number fields $F/k$ and a finite set of
places $S$ of $k$ containing the set $S_\infty$ of all archimedean places. We also fix an irreducible
finite dimensional complex character $\chi$ of $G := G_{F/k}$ and a subfield $E$ of $\C$
that is both Galois and of finite degree over $\Q$ and over which $\chi$ can be
realised. We write $\check \chi$ for the contragredient of $\chi$ and $e_\chi$ for the primitive central
idempotent $\chi(1)|G|^{-1}\sum_{g\in G}\check\chi(g)g$ of $E[G]$. We set $\Gamma := G_{E/\Q}$
and write $\mathcal{O}$ for the ring of integers in $E$. For any $G$-module $M$ we write the
natural semi-linear action of $\Gamma$ on $E\otimes M$ as $(\gamma,x) \mapsto x^\gamma$
(so $x^\gamma = \gamma(e)\otimes m$ if $x = e\otimes m$ with $e \in E$ and $m \in M$).

We fix a strictly negative integer $r$ and for
any field $L$ we write the higher algebraic $K$-group $K_{1-2r}(L)$ additively. We
consider the composite homomorphism
\begin{equation}\label{reg def}
   \reg_{1-r}: K_{1-2r}(\C) \to H_{\mathcal D}^1(\textup{Spec}(\C) , \R(1-r)) = \C/\R(1-r) 
\buildrel{\simeq}\over{\rightarrow} \R(-r)
\end{equation}
where the first arrow is the Beilinson regulator map and the second
is the isomorphism induced by the decomposition $\C = \R(1-r)\oplus \R(-r)$.
 We let $\tau$ be the non-trivial element in $G_{\C/\R}$ and use the same symbol to denote the induced involution on $ K_{1-2r}(\C)$.
 We recall that $\reg_{1-r}\circ\tau = (-1)^r \reg_{1-r}$. We note that the index $ 1-r $ corresponds to the only Adams weight in algebraic $ K $-theory
on which $ \reg_{1-r} $ is non-trivial.  In particular $ \reg_{1-r} $ is trivial on the image
in $K_{1-2r}(\C)$ of the Milnor $K$-group $K^M_{1-2r}(\C)$ since this has weight $ 1-2r $.

We regard the set $\emb{F}$ of embeddings $F\to \C$ as a left
$G\times G_{\C/\R}$-module by setting  $(g\times \omega)(\sigma) = \omega\circ \sigma\circ g^{-1}$
for each $g\in G, \omega\in G_{\C/\R}$ and $\sigma\in \emb{F}$. For each
$\sigma\in \emb{F}$ we write
${\rm reg}_{1-r,\sigma}: \mathcal{O}\otimes K_{1-2r}(F) \to \C$ for the
$\mathcal{O}$-linear map sending each $e\otimes x$ with $e\in\mathcal{O}$ and
$x\in K_{1-2r}(F)$ to $e\cdot{\rm reg}_{1-r}(\sigma(x))$, where $\sigma$
also denotes the induced homomorphism $K_{1-2r}(F) \to K_{1-2r}(\C)$.
We write $\tau_\sigma$ for the generator of the decomposition
subgroup $D_\sigma$ in $G$ of the place of $F$ that corresponds to $\sigma$.

For each $\gamma$ in $\Gamma$ we write $L'_S(r,\chi^\gamma)$ for the value at $s = r$ of
the first derivative of the $S$-truncated Artin $L$-function
$L_S(s,\chi^\gamma)$ of $\chi^\gamma$. In the case $S = S_\infty$
we often abbreviate $L_S(s,\chi^\gamma)$ and $L'_S(r,\chi^\gamma)$ to $L(s,\chi^\gamma)$ and $L'(r,\chi^\gamma)$ respectively.
Finally we set $w_r(\chi):= w_{1-r}(F^{{\rm ker}(\chi^\gamma)})^{\chi^\gamma(1)}$ for any
$\gamma\in \Gamma$, and we note that this number is indeed independent of the choice of $\gamma$.

We can now state the central conjecture of this article.

\begin{conj}\label{chinlike}
Assume that $L(r,\check\chi) =0$. For each $\sigma$ in $\emb{F}$ there exists an element
$\epsilon_\sigma(\chi,S)$ of $\mathcal{O}\otimes K_{1-2r}(F)$ with the following property: for all
$\sigma'\in \emb{F}$ and all $\gamma\in \Gamma$ one has
\begin{equation}\label{conj eq}
 (2\pi i)^r{\reg}_{1-r,\sigma'}(\epsilon_\sigma(\chi,S)^\gamma)
=
 \wr \gamma(\chic{\chi}{\sigma',\sigma}) L'_S(r,\check\chi^\gamma)
\end{equation}
with
\[ \chic{\chi}{\sigma',\sigma} := \begin{cases} \check\chi(g) +(-1)^{r}\check\chi(g\tau_\sigma) &\text{if $\sigma(k) \subset \R$ and
$\sigma' = g(\sigma)$ for some $g \in G$},\\
\check\chi(g) &\text{if $\sigma(k) \nsubseteq \R$ and $\sigma' =
g(\sigma)$ for some $g \in G$},\\
(-1)^r\check\chi(g) &\text{if $\sigma(k) \nsubseteq \R$ and $\sigma'
= \tau\circ g(\sigma)$ for some $g \in G$},\\
0 &\text{otherwise.}\end{cases}
\]
\end{conj}
\smallskip

\subsection{}We make several straightforward remarks about Conjecture~\ref{chinlike}.

\begin{remark}\label{firstremark} \hfill

(i)
If Conjecture~\ref{chinlike} holds for $\sigma$ and $\chi$, then it holds for
$\sigma$ and $\chi^\delta$ for any $\delta$ in $\Gamma$
with $ \epsilon_\sigma(\chi^\delta,S) = \epsilon_\sigma(\chi,S)^\delta $.

(ii)
If Conjecture~\ref{chinlike} holds for $\sigma$ and $\chi$, then it holds for
$\tau\circ\sigma$ and $\chi$ with $ \epsilon_{\tau\circ\sigma}(\chi,S) =(-1)^r \epsilon_\sigma(\chi,S) $.

(iii)
If Conjecture~\ref{chinlike} holds for $\sigma$ and $\chi$, and
$ h $ is in $ G $, then it holds for $ h(\sigma) $ and $ \chi $
with $\epsilon_{h(\sigma)}(\chi,S) = h(\epsilon_\sigma(\chi,S))$. This is true because
 both $\reg_{1-r,\sigma'}(h(\epsilon_\sigma(\chi,S))) = \reg_{1-r,h^{-1}(\sigma')}(\epsilon_\sigma(\chi,S))$ and
$ \chic{\chi}{h^{-1}(\sigma'),\sigma} = \chic{\chi}{\sigma',h(\sigma)} $
(since $\tau_{h(\sigma)} = h \tau_\sigma h^{-1}$ if $ \sigma(k)\subset\R $
and $\check\chi$ is a class function on $G$).

(iv)
Any element $\epsilon_\sigma(\chi,S)$ that satisfies (\ref{conj eq}) for all $\sigma'$ and any
fixed $\gamma$ is necessarily unique modulo $\mathcal{O}\otimes K_{1-2r}(F)_{\rm tor}$
(see the proof of Proposition~\ref{func props}(iv)(a) below).
\end{remark}

\begin{remark}\label{chin=2}
If $L'(r,\check\chi) =0$, then $L'(r,\check\chi^\gamma) =0$ for all $\gamma\in \Gamma$ and so Conjecture~\ref{chinlike} is interesting only if $L(r,\check\chi) = 0 \not= L'(r,\check\chi)$.
By computing orders of vanishing via an explicit analysis of the Gamma factors and functional equation
of $L(s,\check\chi)$ (as in \cite[Chap.\ VII, {\S}12]{Neu}), one finds that this is the case in precisely the following situations:
\begin{itemize}
\item[(i)] $k$ has exactly one complex place, $\chi(1) = 1$ and
$\chi(1) + (-1)^r \chi(\tau_{\sigma'}) = 0$ for all $\sigma'\in \emb{F}$ with $\sigma'(k) \subset\R$;
\item[(ii)] $k$ is totally real, $\chi(1) + (-1)^r\chi(\tau_\sigma) =2$ for some $\sigma\in \emb{F}$
and $\chi(1) + (-1)^r\chi(\tau_{\sigma'}) =0$ for all $\sigma'\in \emb{F}\setminus \{g(\sigma): g \in G\}$.
\end{itemize}
In both of these situations $\chi$ is also realisable over its character field $\Q(\chi)$.
Indeed, this is
obvious in case~(i) and in case~(ii) can be shown
as follows.
For any $\sigma$ in $\emb{F}$ and any $\chi$,
we let $e_{\sigma,\chi}$ in $\Q(\chi)[G]$ denote the idempotent
$\frac{1}{2}(1+(-1)^r\tau_\sigma)e_\chi$
if $\sigma(k)\subset\R$ and the central idempotent $e_\chi$ otherwise.
In case~(ii) above we consider $e_{\sigma,\chi}$ for the given $\sigma$.
Then the character of the
$\Q(\chi)[G]$-module $V:=\Q(\chi)[G]e_{\sigma,\chi}$ is equal to
$m\chi$ for some natural number $m$ and we
must show $m=1$. But in $\C[G]e_{\sigma,\chi} \cong M_{\chi(1)}(\C)$ the idempotent $e_{\sigma,\chi}$
identifies with a matrix of rank one (since
$\chi(1) + (-1)^r\chi(\tau_\sigma) =2$) and so
$\chi(1) = {\rm dim}_\C(\C[G]e_{\sigma,\chi}) = {\rm dim}_\C(\C\otimes_{\Q(\chi)}V) = m\chi(1)$,
as required.
\end{remark}


There are a small number of cases in which the elements predicted by Conjecture \ref{chinlike} can be explained theoretically.

If $k = \Q$ then Remark~\ref{chin=2}(ii) implies that $L(r,\check\chi) = 0 \not= L'(r,\check\chi)$ if and only if $\chi(\tau_{\sigma}) = (-1)^r(2-\chi(1))$ for all $\sigma\in \emb{F}$, which certainly holds if $\chi(1) = 1$ and $\chi(\tau_{\sigma}) = (-1)^r$ for all $\sigma$.
In this case Proposition~\ref{func props}(ii) below
also allows us to assume that $F$ is an Abelian extension of $\Q$ and so one can explicitly construct elements $\epsilon_\sigma(\chi,S)$ as in Conjecture~\ref{chinlike} by using the cyclotomic elements in algebraic $K$-theory of Deligne et al.\
(see also the proof of Corollary~\ref{ab case} below).

If $k$ is imaginary quadratic, then Remark~\ref{chin=2}(i) implies that $L(r,\check\chi) = 0 \not= L'(r,\check\chi)$ if and only if $\chi(1)=1$ and in any such case Deninger has proved Gross's Conjecture (in \cite[Th. 3.1]{deninger}). It seems likely that Deninger's methods can be used to directly construct the corresponding elements in Conjecture \ref{chinlike} both for any such $\chi$ and, when $F$ is Galois over $\Q$, also for the character ${\rm Ind}^{G_{F/\Q}}_{G}(\chi)$ if this is irreducible. For example, interesting work in this direction (which does not however bound the denominators of the elements that arise) is described by Levin in \cite{Lev}.

In general, however, we are not aware of any theoretical constructions that can account for the existence of the
elements $\epsilon_\sigma(\chi,S)$ in Conjecture \ref{chinlike} for other classes of characters $\chi$. In \S\ref{db2} we will also
see that, for certain sets $S$, these mysterious elements should encode information about the
explicit module structure of even degree algebraic
$K$-groups in a manner that is strikingly parallel to the way in
which cyclotomic units encode structural information about ideal
class groups (see Remark~\ref{rubin remark}).

\section{Reductions and reformulations}

\subsection{}
Conjecture~\ref{chinlike} admits several useful reformulations
that we shall describe below.
We begin by recording some of its properties.
Here we write $\overline{\epsilon}_\sigma(\chi,S)$ for the image in $E\otimes
K_{1-2r}(F)$ of the conjectural element $\epsilon_\sigma(\chi,S)$.

\begin{prop}\label{func props} \hfill
\begin{itemize}
\item[(i)] It is enough to verify Conjecture~\ref{chinlike} with $S=S_\infty$.
\item[(ii)] It is enough to verify Conjecture~\ref{chinlike} after replacing $F$ by $F^{\ker(\chi)}$.
\item[(iii)] It is enough to verify Conjecture~\ref{chinlike} when $L'(r,\check\chi) \not= 0$ and with $E = \Q(\chi)$.
\item[(iv)] Assume that Conjecture~\ref{chinlike} is valid for $\sigma$ in $\emb{F}$.
\begin{itemize}
\item[(a)] $\overline{\epsilon}_\sigma(\chi,S)$ is uniquely specified by the formulas in Conjecture~\ref{chinlike}.
\item[(b)]
$\overline{\epsilon}_\sigma(\chi,S) = e_\chi \overline{\epsilon}_\sigma(\chi,S) $;
if $\sigma(k) \subset \R$,
then $\tau_\sigma(\overline{\epsilon}_\sigma(\chi,S)) = (-1)^r\overline{\epsilon}_\sigma(\chi,S)$.
\end{itemize}
\end{itemize}
\end{prop}

\begin{proof}
If Conjecture~\ref{chinlike} is valid as stated, then for any $v\notin S$
it is validated with $S$ replaced by $S' = S\cup\{v\}$ by setting $\epsilon_\sigma(\chi,S'):= {\rm det}_\C(1-\sigma_w{\rm N}v^{-r}\mid V_\chi^{I_w})\epsilon_\sigma(\chi,S)$ where $w$ is any
place of $F$ above $v$, $G_w$ and $I_w$ are the decomposition and inertia subgroups of $w$ in $G$, $\sigma_w$ the
Frobenius automorphism in $G_w/I_w$ and $V_\chi$ a $\C[G]$-module of character $\chi$. (Here we use the fact that,
since $r<0$, the element ${\rm det}_\C(1-\sigma_w{\rm N}v^{-r}\mid V_\chi^{I_w})$ belongs to $\mathcal{O}$.) This
proves claim (i).

Claim (ii) is easily checked and the first assertion of (iii) is obvious since if $L'(r,\check\chi) = 0$,
then $L'(r,\check\chi^\gamma)=0$ for all $\gamma \in \Gamma$ and so Conjecture~\ref{chinlike} is validated
by setting $\epsilon_\sigma(\chi,S):= 0$.
Finally, it is clear that if Conjecture~\ref{chinlike} is valid for a given field $E$,
then it is valid for any larger such field and so the second claim of (iii) is true because
$L'(r,\check\chi) \not= 0$ implies that $\chi$ can be realised over $\Q(\chi)$ by Remark~\ref{chin=2}.

To prove claim (iv) we set $\overline{\epsilon}_\sigma := \overline{\epsilon}_\sigma(\chi,S)$. We extend ${\reg}_{1-r,\sigma'}$ to a homomorphism of $E$-modules $E\otimes
K_{1-2r}(F) \to \C$ in the natural way, write $\widetilde{\reg}_{\sigma'}$ for $(2\pi i)^r{\reg}_{1-r,\sigma'}$ and recall that
$\bigcap_{\sigma'\in \emb{F}}\ker(\widetilde{\reg}_{\sigma'})$ vanishes (by Borel \cite{Borel74}). It is thus
clear that $\overline{\epsilon}_\sigma$ is uniquely specified by (\ref{conj eq})
for all $ \sigma' $ and with $\gamma$ equal to the trivial element, proving (iv)(a).
For the same
reason, to prove (iv)(b) it suffices to show that
$\widetilde{\reg}_{\sigma'}(e_\chi \overline{\epsilon}_\sigma) = \widetilde{\reg}_{\sigma'}(\overline{\epsilon}_\sigma)$
and, if $\sigma(k) \subset\R$, that $\widetilde{\reg}_{\sigma'}(\tau_\sigma(\overline{\epsilon}_\sigma)) =
(-1)^r\widetilde{\reg}_{\sigma'}(\overline{\epsilon}_\sigma)$,
for all $\sigma'$.  But
$
 \widetilde{\reg}_{\sigma'}(e_\chi\overline{\epsilon}_\sigma)
=
 \chi(1)|G|^{-1}\sum_{h \in G}\check\chi(h)\widetilde{\reg}_{h^{-1}(\sigma')}(\overline{\epsilon}_\sigma)
=
 \widetilde{\reg}_{\sigma'}(\overline{\epsilon}_\sigma)
$
where the last equality follows by combining the conjectural equality (\ref{conj eq}) (with $\gamma$
trivial) with the fact that
$\chi(1)|G|^{-1}\sum_{h \in G}\check\chi(h)\check\chi(h^{-1}g) = \check\chi(g)$
for all $g\in G$, so that
$ \chi(1)|G|^{-1}\sum_{h\in G} \check\chi(h)\chic{\chi}{h^{-1}(\sigma'),\sigma}=\chic{\chi}{\sigma',\sigma}$.
Moreover, if $\sigma(k) \subset \R$ then one easily checks that
$ \chic{\chi}{\tau_\sigma^{-1}(\sigma'), \sigma} = (-1)^r \chic{\chi}{\sigma',\sigma} $,
and combining this with (\ref{conj eq}) (with $ \gamma $ trivial)
we find that $\widetilde{\reg}_{\sigma'}(\tau_\sigma(\overline{\epsilon}_\sigma))=
\widetilde{\reg}_{\tau_\sigma^{-1}(\sigma')}(\overline{\epsilon}_\sigma)=
(-1)^r\widetilde{\reg}_{\sigma'}(\overline{\epsilon}_\sigma)$.
\end{proof}

In the next result we reformulate Conjecture~\ref{chinlike} in a style reminiscent of the
refinement of (a special case of) Stark's conjecture that is stated by
Chinburg in~\cite{chin}.
Here we write $\mathcal{D}_E$ for the different of $E/\Q$.

\begin{prop}\label{extraprop}
Let $ \sigma $ be in $ \emb{F} $.

\begin{itemize}
\item[(i)]
If Conjecture~\ref{chinlike} holds for $ \sigma $,
then for each $d$ in $\mathcal{D}_E^{-1}$ the element
\[\qquad{\rm Tr}_d(\epsilon_\sigma(\chi,S)) := \sum_{\gamma\in \Gamma}\gamma(d)\epsilon_\sigma(\chi,S)^\gamma\]
belongs to $K_{1-2r}(F)$ and for all $\sigma' \in \emb{F}$ satisfies
\begin{equation}\label{Trace-id}
\quad
 (2\pi i)^r\reg_{1-r,\sigma'}({\rm Tr}_d(\epsilon_\sigma(\chi,S))) =
 \wr \sum_{\gamma \in\Gamma}\gamma(d)\gamma(\chic{\chi}{\sigma',\sigma})L'_S(r,\check\chi^\gamma)
\,.
\end{equation}
\item[(ii)] If for each $d$ in $\mathcal{D}_E^{-1}$ there exists an element 
$\alpha_{d,\sigma}(\chi,S)$ of $K_{1-2r}(F)$ such that~\eqref{Trace-id}
holds with $\alpha_{d,\sigma}(\chi,S)$ in place of
$ {\rm Tr}_d(\epsilon_\sigma(\chi,S)) $, then Conjecture~\ref{chinlike} holds
for~$ \sigma $.
\end{itemize}
\end{prop}

\begin{proof}
To prove (i) we note Conjecture~\ref{chinlike} implies that for every $\gamma$ in $\Gamma$ one has
$\gamma(d)\epsilon_\sigma(\chi,S)^\gamma\in {\mathcal{D}}_E^{-1}\otimes K_{1-2r}(F)$
and also
$(2\pi i)^r{\reg}_{1-r,\sigma'}(\gamma(d)\epsilon_\sigma(\chi,S)^\gamma) = \wr\gamma(d)\gamma(\chic{\chi}{\sigma',\sigma})L'_S(r,\check\chi^\gamma)$.
It is thus clear that ${\rm Tr}_d(\epsilon_\sigma(\chi,S))$ satisfies~\eqref{Trace-id}.
In addition one has ${\rm Tr}_{E/\Q}({\mathcal{D}}_E^{-1})\subseteq \Z$ and so
\[{\rm Tr}_d(\epsilon_\sigma(\chi,S)) =
 \sum_{\gamma\in \Gamma}(d\epsilon_\sigma(\chi,S))^\gamma \in {\rm Tr}_{E/\Q}(\mathcal{D}_E^{-1})\otimes
K_{1-2r}(F)\subseteq K_{1-2r}(F),\]
as required.  For claim (ii), let $\{d_j\}_j$ be a $\Z$-basis of $\mathcal{D}_E^{-1}$,
$\{e_j\}_j$ the dual $ \Z $-basis of $\mathcal{O}$ with respect to the trace pairing,
and set $\epsilon_\sigma(\chi,S) := \sum_j e_j \otimes \alpha_{d_j,\sigma}(\chi,S) $.
If $ \gamma $ is in $ \Gamma $, then
\begin{alignat*}{1}
  (2\pi i )^r \reg_{1-r,\sigma'}(\epsilon_\sigma(\chi,S)^\gamma)
&=
  \wr \sum_{\gamma',j} \gamma(e_j) \gamma'(d_j) \gamma'(\chic{\chi}{\sigma',\sigma}) L_S'(r,\check\chi^{\gamma'})
\\
&=
  \wr \gamma(\chic{\chi}{\sigma',\sigma}) L_S'(r,\check\chi^\gamma)
\end{alignat*}
because if $\Gamma = \{\gamma_j\}_j$, then
$(\gamma_j(d_i))(\gamma_i(e_j))$ is the identity matrix,
and the same holds for $(\gamma_i(e_j))(\gamma_j(d_i))$.
\end{proof}

\subsection{}
In this subsection we give (in
Proposition~\ref{equivalences}(ii)) a more concise reformulation of Conjecture~\ref{chinlike}
that requires some preliminaries. This reformulation will be particularly useful in Section~\ref{ne} when we numerically verify Conjecture~\ref{chinlike} in the case that  
 $S=S_\infty$ and $L_S'(r,\chi)\not=0$.

Claim ~(iii) of Proposition~\ref{func props} implies that one only needs to consider the case that $L_S(r,\chi)=0\not=L_S'(r,\chi)$.
Remark~\ref{chin=2} specifies when this happens,
and we let $ \embsp{r,\chi}{F}$ denote the subset of $\emb{F} $ comprising
the $2|G|$ elements $\sigma$ with $\sigma(k) \not\subseteq\R$
in case~(i) of this remark, and comprising the $|G|$ elements $\sigma$ with
$\chi(1) + (-1)^r\chi(\tau_{\sigma}) =2$ in case~(ii).
(Equivalently, one has $\sigma\notin \embsp{r,\chi}{F} $ if and only if 
both $\sigma(k)\subset \R$ and $\tau_\sigma$ acts as multiplication by $(-1)^{r+1}$ in the representation underlying~$\chi$.)
Note that $ \embsp{r,\chi}{F} = \embsp{r,\chi^\gamma}{F} $ for $\gamma$ in $\Gamma$.
We recall that the idempotent $e_{\sigma,\chi}$ was defined in Remark~\ref{chin=2}
for any $ \sigma $ in $ \emb{F} $.

\begin{lemma}\label{chispecial}
Assume $L'(r,\check\chi) =0$ and $L'(r,\check\chi) \ne 0$.
\begin{enumerate}
\item[(i)]
If either $\sigma$ or $\sigma'$ is not in $\embsp{r,\chi}{F}$, then $\chic{\chi}{\sigma',\sigma}=0$.

\item[(ii)]
For $\sigma$ in $\emb{F}$, one has $e_{\sigma,\chi}=0$ if and only if $\sigma\notin \embsp{r,\chi}{F}$.

\item[(iii)]
If $g$ is in $G$ and $\sigma$ in $\emb{F}$, then
$\chic{\chi}{\sigma,\sigma}e_{\sigma,\chi} g^{-1} e_{\sigma,\chi} = \chic{\chi}{g(\sigma),\sigma} e_{\sigma,\chi}$.
\end{enumerate}
\end{lemma}

\begin{proof}
For claim~(i) we note that $ \chic{\chi}{\sigma',\sigma} \ne 0 $ implies that 
both $ \sigma $ and $ \sigma' $ are in $ \embsp{r,\chi}{F} $ or neither are.
If $ \sigma \notin \embsp{r,\chi}{F} $ then $ \sigma(k)\subset\R $,
and $ \chic{\chi}{\sigma',\sigma} = 0 $ because of the definitions.
For claim~(ii) we note that $ e_{\sigma,\chi}=0 $ from the definitions
if $ \sigma \notin\embsp{r,\chi}{F} $, and that
$ e_{\sigma,\chi}\ne0 $ if $ \sigma \in\embsp{r,\chi}{F}$: when $ \sigma(k)\not\subseteq\R $
this is obvious, and if $ \sigma(k)\subset\R $ then it follows
by choosing an isomorphism $\Q(\chi)[G]e_\chi \simeq M_{\chi(1)}(\Q(\chi))$ such that
$1+(-1)^r\tau_\sigma$ corresponds to the diagonal matrix with entries $2,0,\dots,0$ on the diagonal.
Claim~(iii) for $\sigma$ in $\embsp{r,\chi}{F}$ follows from similar
considerations using that $\chic{\chi}{\sigma,\sigma}$ is equal to either $1$ or $2$, and for other
$\sigma$ is obvious from~(ii).
\end{proof}

\begin{remark}\label{last one} Lemma~\ref{chispecial}(i) refines Proposition~\ref{func props}
as it implies that Conjecture~\ref{chinlike} holds (with $ \epsilon_\sigma(\chi,S) = 0 $)
when $ L_S'(r,\chi)\not=0 $ and  $ \sigma \notin\embsp{r,\chi}{F}$.
Considering also the first three parts of Remark~\ref{firstremark},
we see that verifying Conjecture \ref{chinlike} for a given $ \chi $ and any given embedding $ \sigma $
in $\embsp{r,\chi}{F} $ does so for all characters in the $\Gamma$-orbit $\{\chi^\gamma: \gamma\in \Gamma\}$
and for all embeddings $ \sigma $ in~$ \emb{F} $.
\end{remark}

In the next result let $K_{1-2r}(F)_{\rm tf}$ denote the image of
$K_{1-2r}(F)$ in $\Q\otimes K_{1-2r}(F)$.
Then $\reg_{1-r,\sigma'}$ factorizes through the map
$\mathcal{O}\otimes K_{1-2r}(F)\to\mathcal{O}\otimes K_{1-2r}(F)_{\rm tf}$
and we use the same notation for the resulting map.

\begin{prop}\label{equivalences}
Let $\sigma$ be in $\emb{F}$.  Then the following statements are equivalent.

\begin{itemize}
\item[(i)]
Conjecture~\ref{chinlike} holds for $\sigma$.

\item[(ii)]
There exists $\beta_\sigma(\chi,S)$ in $\mathcal{O}\otimes K_{1-2r}(F)_{\rm tf} \subseteq E\otimes K_{1-2r}(F)$
that satisfies $e_{\sigma,\chi}\beta_\sigma(\chi,S) = \beta_\sigma(\chi,S)$ and, for all $\gamma$ in $\Gamma$, also
\begin{equation*}
 (2\pi i)^r{\reg}_{1-r,\sigma}(\beta_\sigma(\chi,S)^\gamma)
=
  \wr \chic{\chi}{\sigma,\sigma} L'_S(r,\check\chi^\gamma)
\,.
\end{equation*}

\end{itemize}
\end{prop}

\begin{proof}
That (i) implies (ii) can be seen from Proposition~\ref{func props}(iv)(b): we may simply set
$\beta_\sigma(\chi,S) := \overline{\epsilon}_\sigma(\chi,S)$.
For the converse, we note that Conjecture~\ref{chinlike} always holds
if $ L_S'(r,\check\chi)=0 $, or if $ L_S'(r,\check\chi)\ne0 $
and $ \sigma \notin\embsp{r,\chi}{F} $,
so we assume that both $ L_S'(r,\check\chi)\ne0 $ and $ \sigma \in \embsp{r,\chi}{F} $.
Then $ \chic{\chi}{\sigma,\sigma}$ is equal to either $1 $ or $ 2 $ and so
it will suffice to show that
\begin{equation}\label{easyeq}
 \chic{\chi}{\sigma,\sigma} \reg_{1-r,\sigma'}(\beta_\sigma(\chi,S)^\gamma)
=
 \gamma(\chic{\chi}{\sigma',\sigma}) \reg_{1-r,\sigma}(\beta_\sigma(\chi,S)^\gamma)
\end{equation}
for all $ \sigma' \in\emb{F} $ and all $ \gamma\in\Gamma $,
because then~\eqref{conj eq} holds
for any lift $ \epsilon_\sigma(\chi,S)$ of $\beta_\sigma(\chi,S) $
to $ \mathcal{O}\otimes K_{1-2r}(F) $.

If $\sigma'\notin \embsp{r,\chi}{F},$ then $\chic{\chi}{\sigma',\sigma}=0$ by Lemma~\ref{chispecial}(i).
On the other hand, for $ \gamma \in \Gamma $ we have $ \beta_\sigma(\chi,S)^\gamma = e_{\chi^\gamma} \beta_\sigma(\chi,S)^\gamma $ because
 $ \beta_\sigma(\chi,S) = e_\chi \beta_\sigma(\chi,S) $, and
\begin{equation*}
 \reg_{1-r,\sigma'}(\beta_\sigma(\chi,S)^\gamma)
=
 \reg_{1-r,\sigma'}(\tfrac{1+(-1)^r\tau_{\sigma'}}{2}\beta_\sigma(\chi,S)^\gamma)
=
 \reg_{1-r,\sigma'}(e_{\sigma',\chi^\gamma}\beta_\sigma(\chi,S)^\gamma)
\,,
\end{equation*}
which is trivial by Lemma~\ref{chispecial}(ii) because $ \embsp{r,\chi}{F} = \embsp{r,\chi^\gamma}{F} $,
 thus establishing~\eqref{easyeq} for such $ \sigma' $.

If now $\sigma'\in \embsp{r,\chi}{F} $, then we distinguish
three cases: (a) $\sigma(k)\not\subseteq\R$ and $\sigma'=g(\sigma)$ for some $g$ in $G$;
(b) $\sigma(k)\not\subseteq\R$ and $\sigma'=\tau\circ g(\sigma)$ for some $g$ in $G$;
(c) $\sigma(k)\subset\R$ and $\sigma'=g(\sigma)$ for some $g$ in $G$.
In case (a) $e_{\sigma,\chi}$ is equal to the central idempotent $e_\chi$,
$\chic{\chi}{\sigma,\sigma}=\chi(1)=1$,
and for $\gamma\in \Gamma$ we have
$ g^{-1}\beta_\sigma(\chi,S)^\gamma= e_{\chi^\gamma}g^{-1}e_{\chi^\gamma}\beta_\sigma(\chi,S)^\gamma
 = \gamma(\chic{\chi}{g(\sigma),\sigma})\beta_\sigma(\chi,S)^\gamma $ by Lemma~\ref{chispecial}(iii).
Now~\eqref{easyeq} follows because
$ \reg_{1-r,g(\sigma)}(\beta_\sigma(\chi,S)^\gamma) = \reg_{1-r,\sigma}(g^{-1}\beta_\sigma(\chi,S)^\gamma) $.
Case (b) is dealt with similarly using the equalities $\reg_{1-r,\sigma'} = (-1)^r \reg_{1-r,g(\sigma)}$
and $ \chic{\chi}{\sigma',\sigma} = (-1)^r \chic{\chi}{g(\sigma),\sigma} $.
In case (c) one has $ \chic{\chi}{\sigma,\sigma} = 2 $
and the left hand side of~\eqref{easyeq} is equal to
\begin{equation*}
 2 \reg_{1-r,\sigma}(g^{-1}\beta_\sigma(\chi,S)^\gamma)
=
 \reg_{1-r,\sigma}((1+(-1)^r\tau_\sigma)g^{-1}\beta_\sigma(\chi,S)^\gamma)
\,.
\end{equation*}
Using that
$
 g^{-1}\beta_\sigma(\chi,S)^\gamma
=
 g^{-1} e_{\sigma,\chi^\gamma}\beta_\sigma(\chi,S)^\gamma
=
 e_{\chi^\gamma} g^{-1} e_{\sigma,\chi^\gamma}\beta_\sigma(\chi,S)^\gamma
$
for the central idempotent $e_{\chi^\gamma}$
we see from Lemma~\ref{chispecial}(iii) that~\eqref{easyeq}
again holds.
\end{proof}

\subsection{}
We conclude this section with some more results of independent
interest, which will also be used in the sequel.

\begin{lemma}\label{ord=mult}
If the irreducible character $ \chi $ is realisable over $ E $,
then the multiplicity of the corresponding representation in $ E\otimes K_{1-2r}(F) $
 is equal to the order of vanishing of $ L(s,\check\chi) $ at~$ s=r $.
\end{lemma}

\begin{proof} For each integer $m$ we define a $(G\times G_{\C/\R})$-module $B_{m} := \bigoplus_{\emb{F}}(2\pi i)^{-m}\Z$, where $G$ acts on $\emb{F}$ in the way specified just before Conjecture~\ref{chinlike} and $G_{\C/\R}$ acts diagonally (on both $\emb{F}$ and $(2\pi i)^{-m}\Z$). Then, according to Borel's theorem \cite{Borel74}, the $G$-invariant pairing
\begin{equation}\label{Borel pairing}  \Q\otimes B_{-r}^{G_{\C/\R}} \times \Q\otimes K_{1-2r}(F) \to \R \end{equation}
that is induced by mapping each element $ ((2\pi i)^r \sigma,\alpha) $ to $ (2\pi i)^r\reg_{1-r}(\sigma(\alpha)) $
is non-degenerate.

If we now extend coefficients to $ E $, then the non-degeneracy of (\ref{Borel pairing}) implies that for
any idempotent $ \pi $ in $ E[G] $ one has ${\rm dim}_E(\pi (E\otimes K_{1-2r}(F))) ={\rm dim}_E(\check\pi(E\otimes B_r))$, where $\check\pi$ is the image of $\pi$ under the $E$-linear anti-involution of $E[G]$ that inverts elements of $G$. Finally we note that if $ \pi =e_\chi$, so $ \check\pi = e_{\check\chi} $, then the description of the Gamma factors and functional equation
of $L(s,\check\chi)$ (as in \cite[Chap.\ VII, {\S}12]{Neu}) implies that $\chi(1)^{-1}{\rm dim}_E(\check\pi(E\otimes B_r))$ is equal to the vanishing order of $ L(s,\check\chi) $ at $ s=r $.
\end{proof}

\begin{remark}\label{useful remark}
There are several ways in which the observations made above are useful
when making numerical investigations of Conjecture~\ref{chinlike}.

(i)
If $L_S(r,\check\chi) = 0 \ne L_S'(r,\check\chi)$, then
by Proposition~\ref{func props}(iv)(b),
$\overline{\epsilon}_\sigma(\chi,S)$
belongs to the $\mathcal{O}$-sublattice $e_{\sigma,\chi}(\mathcal{O}\otimes K_{1-2r}(F)_{\rm tf})$ of
$E\otimes K_{1-2r}(F)$.  If $\sigma$ is in $\embsp{r,\chi}{F}$
then this $\mathcal{O}$-lattice has rank one by Lemma~\ref{ord=mult} and
the proof of Lemma~\ref{chispecial}(ii), and this provides a very strong restriction on where one searches to find $\overline{\epsilon}_\sigma(\chi,S)$.
This also applies to $\beta_\sigma(\chi,S)$ in Proposition~\ref{equivalences},
and implies that $ \beta_\sigma(\chi,S) $ is unique in this case.

(ii)
If $r$ is even, then Proposition~\ref{func props}(ii) and (the second part of) (iv)(b) combine to imply that the element $e|K_{1-2r}(F^{\ker(\chi)})_{\rm tor}|{\epsilon}_\sigma(\chi,S,d)$ belongs to the
image of $K_{1-2r}(F^{\ker(\chi)D_\sigma})$ in $K_{1-2r}(F)$, where $e\in \{1,2\}$ is the exponent of the
Tate cohomology group $\hat H^0(\ker(\chi)D_\sigma/\ker(\chi),K_{1-2r}(F^{\ker(\chi)}))$.  This observation
is useful because it can be computationally much easier to search for elements in $K_{1-2r}(F')$
for proper subfields $F'$ of $F$ rather than in $K_{1-2r}(F)$ itself.

(iii) The first observation made in the proof of Lemma \ref{ord=mult} implies that if $F$ has signature $[r_1,r_2]$, then the rank of $K_{1-2r}(F)$ is equal to
 $r_1+r_2$ if $r$ is even and to $r_2$ if $r$ is odd. This explicit (and well-known) formula will be useful in the sequel.

\end{remark}

\begin{remark}\label{QabFremark}
It is clear that any commutator in $ G_{\Q^c/\Q} $ acts trivially on $ \mu_a(F) = H^0(G_{\Q^c/F},\Q/\Z(a))$.
Hence one has $ \mu_a(F) = \mu_a(F\cap\Q^{\rm ab}) $,
where $ \Q^{\rm ab} $ is the maximal Abelian extension of $ \Q $.
If $ a $ is even, then one sees similarly that $ \Q^{\rm ab} $
may be replaced by its maximal totally real subfield.
\end{remark}

\section{Theoretical evidence}\label{db2}

In this section we describe the connection between Conjecture \ref{chinlike} and the modified Lichtenbaum-Gross Conjecture of Chinburg et al. \cite{ckps} and thereby deduce the validity of Conjecture \ref{chinlike} in some important special cases. We also show that the elements predicted by Conjecture \ref{chinlike} should encode information about the structure (as Galois modules) of certain even degree higher algebraic $K$-groups.

\subsection{}We first state the main result of this section (which will be proved in \S\ref{proof main}).

\begin{thm}\label{mainres}
We assume that $S$ contains $S_\infty$ and all places that ramify in $F/k$. We also assume that the modified Lichtenbaum-Gross Conjecture (see Conjecture~\ref{ssc} below) is valid for $\chi$.
\begin{itemize}
\item[(i)] Then Conjecture~\ref{chinlike} is valid.
\item[(ii)] Further, the element ${\rm Tr}_1(\epsilon_\sigma(\chi,S)) := \sum_{\gamma \in \Gamma}\epsilon_\sigma(\chi,S)^\gamma$ belongs to
$K_{1-2r}(F)$ and for all $\phi$ in $\Hom_G(K_{1-2r}(F),\Z[G])$ one has
\[ \frac{|G|^{2}}{\chi(1)}\phi({\rm Tr}_1(\epsilon_\sigma(\chi,S)))\in {\rm Ann}_{\Z[G]}(\bigoplus_{p\not= 2}H^2_{{\rm \acute et}}(\Spec (\mathcal{O}_{F,S}[\frac{1}{p}]),\Z_p(1-r))).\]
\end{itemize}
\end{thm}

\begin{cor}\label{ab case}
Fix $k = \Q$ and assume that $\chi(1) =1$.
\begin{itemize}
\item[(i)] Then Conjecture~\ref{chinlike} is valid for any set $S$ as in Theorem~\ref{mainres}.
\item[(ii)] The assertion of Theorem~\ref{mainres}(ii) is also valid for any such set $S$.
\end{itemize}
\end{cor}

\begin{proof} In view of Theorem~\ref{mainres}, Proposition~\ref{func props}(ii) and Remark~\ref{ssc rem}(ii) below it is enough to recall that if $F$ is an Abelian extension of $\Q$, then the equivariant Tamagawa number conjecture of \cite[Conj. 4(iv)]{bufl00} is valid for the pair $(h^0(\Spec (F))(r),\Z[G])$. Indeed, this case of the equivariant Tamagawa number conjecture is proved by Flach and the first author in \cite[Cor. 1.2]{bufl05} (with corrections to the $2$-primary part of the argument of
loc.\ cit.\ recently provided by Flach in \cite{flach}). \end{proof}

\begin{remark}\label{rubin remark}
We write $\mathcal{O}_{F,S}$ for the subring of $F$ comprising elements that are
integral at all places that do not lie above a place in $S$. Then it is conjectured by Quillen and Lichtenbaum that for all
odd primes $p$ the
natural Chern class homomorphism $\Z_p\otimes
K_{-2r}(\mathcal{O}_{F,S}) \rightarrow H^2_{{\rm \acute et}}(\Spec
(\mathcal{O}_{F,S}[\frac{1}{p}]),\Z_p(1-r))$ is bijective. In the case $r = -1$ this conjecture has been proved by Tate in \cite{tateK2}. (It is also known, by work of Suslin, that the Quillen-Lichtenbaum Conjecture is a consequence of a conjecture of Bloch and Kato relating Milnor $K$-theory to \'etale cohomology and it is widely believed that recent fundamental work of Voevodsky and Rost has led to a proof of this conjecture of Bloch and Kato.) Whenever the
 conjecture of Quillen and Lichtenbaum is valid the module $\bigoplus_{p\not= 2}H^2_{{\rm \acute et}}(\Spec
(\mathcal{O}_{F,S}[\frac{1}{p}]),\Z_p(1-r))$ can be replaced by
$\Z[\frac{1}{2}]\otimes K_{-2r}(\mathcal{O}_{F,S})$ in the statement
of Theorem~\ref{mainres}(ii). In the setting of Corollary~\ref{ab
case} we thereby obtain a rather striking analogue of the result of Rubin
in \cite[Th. (2.2) and the following Remark]{rubininv} concerning
the annihilation of ideal class groups in absolutely Abelian
fields.
\end{remark}

\begin{remark}\label{nickel remark} We recently learnt Nickel has shown that for each strictly negative integer $r$ the equivariant Tamagawa number conjecture for $(h^0({\rm Spec}(F))(r),\Z[G])$ implies (via the reinterpretation given in \cite[Prop. 4.2.6]{ltav}) that certain elements constructed from the leading term at $s=r$ of truncated Artin $L$-functions should belong to $\Z[G]$ and annihilate $\bigoplus_{p\not= 2}H^2_{{\rm \acute et}}(\Spec (\mathcal{O}_{F,S}[\frac{1}{p}]),\Z_p(1-r))$ (cf. \cite[Th. 4.1]{Nickel}). His prediction is however much less explicit than that given in Theorem \ref{mainres}(ii) above and it would be interesting to know if there is any direct link between them.
\end{remark}


\subsection{}
In this subsection we discuss some necessary algebraic preliminaries.

We fix $G$ and $\chi$ as in \S\ref{soc sect} and, following Proposition~\ref{func props}(iii), we set $E:= \Q(\chi)$.
 We write $\sum_{j = 1}^{j = \chi(1)}f^j_{\chi}$ for a decomposition of
$e_{\chi}$ as a sum of mutually orthogonal indecomposable idempotents in
$E[G]$. We write $\mathcal{O}$ for the ring of
algebraic integers in $E$ and for each $j$ we choose a maximal $\mathcal{O}$-order
${\mathfrak M}^j$ in $E[G]$ that contains $f^j_{\chi}$. For any character $\psi = \chi^\gamma$, with $\gamma\in \Gamma$, we set $e_\psi := \psi(1)|G|^{-1}\sum_{g\in G} \check\psi(g)g = (e_\chi)^\gamma$ and $f^j_{\psi} := (f^j_{\chi})^\gamma$ and define an $\mathcal{O}$-torsion-free right $\mathcal{O}[G]$-module $T^j_{\psi} := f^j_{\psi}({\mathfrak M}^j)^\gamma = (f^j_{\chi} {\mathfrak M}^j)^\gamma$. The associated right $E[G]$-module $V^j_{\psi} := E\otimes_\mathcal{O}T^j_{\psi}$ has character $\psi$. For any (left) $G$-module $M$ we define a module
$M^j[\psi] := T^j_{\psi}\otimes M$ upon which each $g$ in $G$ acts on the left
by $t\otimes m\mapsto tg^{-1}\otimes g(m)$ for each $t$ in $T^j_{\psi}$ and $m$ in $M$. Then there is a natural isomorphism of (left) $\mathcal{O}[G]$-modules
\begin{equation}\label{translate}
  M^j[\psi] \cong \Hom_{\mathcal{O}}(T^{j,*}_{\psi},\mathcal{O}\otimes M)
\end{equation}
where $G$ acts in the usual (diagonal) manner on the $\Hom$-set
and the module $T^{j,*}_{\psi} := \Hom_{\mathcal{O}}(T^j_{\psi},\mathcal{O})$ is endowed
with the natural left $G$-action and hence spans a left
$E[G]$-module of character $\psi$.

For any subgroup $J$ of $G$ we write $M^J$, resp. $M_J$, for
the maximal submodule, resp. quotient, of $M$ upon
which $J$ acts trivially. In particular we obtain a left, resp.
right, exact functor
 $M \mapsto M^{j,\psi}$, resp. $M \mapsto M^j_\psi$,
 from the category of left $G$-modules to the category of
$\mathcal{O}$-modules by setting $M^{j,\psi} := M^j[\psi]^G$ and $M^j_\psi
:= M^j[\psi]_G \cong T^j_\psi\otimes_{\Z [G]}M$. 
 By replacing $T^j_\psi$ by $\Z_p\otimes T^j_\psi$ we extend the notation $M^{j,\psi}$ and $M^j_\psi$ to $\Z_p[G]$-modules $M$ in the obvious way. We will often use the fact that the action of  $\sum_{g\in G}g$ on $M^j[\psi]$ induces a homomorphism of $\mathcal{O}$-modules $\,t^j(M,\psi):
  M^j_\psi\xrightarrow{}M^{j,\psi}$ that has finite kernel and finite cokernel.

The module $E[G]\otimes M$ has two commuting
left actions of $G$: the first via left multiplication on $E[G]$ and
the second such that each $g$ in $G$ sends $x\otimes m$ to
$xg^{-1}\otimes g(m)$ for $x$ in $E[G]$ and $m$ in $M$. We write
$(E[G]\otimes M)^{G,2}$ for the subset of
$E[G]\otimes M$ comprising elements that are invariant under the second action of $G$ and use the first action of $G$ on
$E[G]\otimes M$ to regard $(E[G]\otimes M)^{G,2}$ as an $E[G]$-module. If $M$ is finitely generated and
torsion-free, then we always regard $M^{j,\psi}$ as an $\mathcal{O}$-submodule of
$\mathcal{O}\otimes M$ by means of the identification described in the following result.

\begin{lemma}\label{jacres}\hfill
\begin{itemize}
\item[(i)] The $E$-linear map $E[G]\otimes M \to E\otimes M$ that sends
$g\otimes m$ to $g(m)$ for each $g$ in $G$ and $m$ in $M$
restricts to give an isomorphism of $E[G]$-modules $\iota:(E[G]\otimes
M)^{G,2} \cong
 E \otimes M$. 
 
\item[(ii)] Assume now that $M$ is finitely generated and
torsion-free and for each character $\psi = \chi^\gamma$, with $\gamma\in \Gamma$, set ${\rm pr}_\psi := \sum_{g \in G}\check\psi(g)g
\in \mathcal{O}[G]$. Then one has $\,{\rm pr}_\psi(\mathcal{O}\otimes M) \subseteq \psi(1)^{-1}\sum_{j=1}^{j=\psi(1)}\iota(M^{j,\psi})\subseteq \mathcal{O}\otimes M.$
\end{itemize}
\end{lemma}

\begin{proof} Since $E[G]\otimes M$ is uniquely divisible it is cohomologically trivial with respect to any action of $G$. The
 submodule $(E[G]\otimes
M)^{G,2}$ is thus equal to the $E$-linear span of elements of the form $(\sum_{g \in G}g)(x\otimes m) = \sum_{g \in G}xg^{-1}\otimes g(m)$ with $x$ in $E[G]$ and $m$ in $M$. Using this fact it is straightforward to check that the map $\iota$ is an isomorphism of $E[G]$-modules, as required to prove claim (i).

We next set $n := \psi(1)$ and note that ${\rm pr}_\psi =
n^{-1}|G|e_{\psi}$. Then the first inclusion in claim (ii) is true because for each $m$ in $M$ one has
\begin{alignat*}{1}
 n\cdot{\rm pr}_\psi(1\otimes m) & = |G|e_{\psi}(1\otimes m) = |G|\sum_{j = 1}^{j =n}f_{\psi}^j(1\otimes m) \\
 & = \sum_{j = 1}^{j =n}\iota(\sum_{g \in G}f_{\psi}^j (g^{-1}\otimes g(m)) \in \sum_{j = 1}^{j =n}\iota(M^{j,\psi})
\,.
\end{alignat*}

To prove the second claimed inclusion it suffices to show
$\iota(M^{j,\psi})\subseteq n(\mathcal{O}\otimes M)$ for each index
$j$. To do this we set $N:= \mathcal{O}\otimes M$ and $N^* :=
\Hom_{\mathcal{O}[G]}(N,\mathcal{O}[G])$. Then, as $N$ is finitely
generated and torsion-free, one has $N = \{ x \in
E\otimes_\mathcal{O} N: \theta(x) \in \mathcal{O} [G] \hbox{ for all
} \theta \in N^*\}.$ Thus it suffices to show that
$\theta(M^{j,\psi}) \subseteq n\mathcal{O} [G]$ for all $\theta \in
N^*$. But $\theta(M^{j,\psi}) \subseteq \Z [G]^{j,\psi}$ and, since
$\Z [G]$ is cohomologically trivial and $\Z[G]^j_\psi = T^j_\psi$,
one has $\Z [G]^{j,\psi} = {\rm im}(t^j(\Z [G],\psi)) =
|G|T^j_\psi.$ Thus it suffices to note that $|G|T^j_\psi =
|G|f^j_{\psi}\mathfrak{M}^j\subseteq |G|e_{\psi}\mathfrak{M}^j
\subseteq n\mathcal{O} [G]$, where the latter inclusion follows from
Jacobinski's description in \cite{Jac} of the central conductor of
$\mathfrak{M}^j$ in $\mathcal{O}[G]$ (see also \cite[Th.
(27.13)]{curtisr}).
\end{proof}

For any $G$-module, resp. $\mathcal{O}$-module, $M$ we write
$M_{\rm tf}$ for the image of $M$ in $\Q\otimes M$, resp.
$E\otimes_{\mathcal{O}}M$. We often identify $M_{\rm tf}$ with the quotient $M/M_{\rm tor}$ in the natural way. 
If $M$ is a $G$-module, then for each $m$ in $M$ we write $f^j_{\psi}(m)$ for the image of
$f^j_{\psi}\otimes_{\Z [G]}m \in M^j_{\psi}$ in
$M^j_{\psi,{\rm tf}}$.

\subsection{} In this subsection we recall the modified Lichtenbaum-Gross Conjecture of Chinburg, Kolster, Pappas and Snaith. 
 Since we regard $F/k$ as fixed we set
$C_r := K_{1-2r}(F)$ and also recall the $G\times G_{\C/\R}$-modules $B_m$ that are defined in the proof of Lemma \ref{ord=mult}. We note, in particular, that the non-degeneracy of the pairing (\ref{Borel pairing}) combines with the natural isomorphism of $G\times G_{\C/\R}$-modules $\Hom_\Z(B_{-r},\Z) \cong B_{r}$ to imply that the map $\bigoplus_{\sigma\in \emb{F}}\reg_{1-r,\sigma}$ induces a canonical isomorphism of $\R [G]$-modules
%
\begin{equation*}
 \Regr: \R \otimes C_r \xrightarrow{\sim}\R \otimes B_r^{G_{\C/\R}}.
\end{equation*}
The bijectivity of $\Regr$ then combines with Deuring's Theorem (cf.
\cite[\S6, Exer. 6]{curtisr}) to imply the existence of
 an (in general non-canonical) isomorphism of $\Q[G]$-modules $\,\phi: \Q \otimes B^{G_{\C/\R}}_r \xrightarrow{\sim} \Q \otimes C_r.$ For any such $\phi$, any index $j$ and any character $\psi = \chi^\gamma$, with $\gamma \in \Gamma$, we set
\[ R_r^\phi(\psi) := \mathrm{det}_{\C}[(\C \otimes_\R\Regr)\circ (\C\otimes_\Q\phi)\mid (\C\otimes_E
V^j_{\psi})\otimes _{\Z[G ]}B_r^{G_{\C/\R}}] \in \C^\times.\]
(This element doesn't depend on $j$ because each $E[G]$-module $V_\psi^j$ has the same character $\psi$.)

We now fix a finite set $S$ of places of $k$ that contains $S_\infty$ and all places that ramify in $F/k$. Then the
`modified Lichtenbaum-Gross Conjecture' predicts an explicit Euler
characteristic formula for the $\mathcal{O}$-module generated by $
L_S^*(r,\check\psi)/R^\phi_r(\psi)$ where $L_S^*(r,\check\psi)$ is
the leading non-zero coefficient in the Taylor expansion at $s= r$
of the function $L_S(s,\check\psi)$. Before stating this conjecture
we must recall a useful auxiliary result. In this result we set $\Z'
:= \Z[\frac{1}{2}]$ and for each odd prime $\ell$ we write
 $\,{\rm ch}^1_{\ell,1-r}: \Z_\ell\otimes
K_{1-2r}(\mathcal{O}_{F,S}) \rightarrow H^1_{{\rm \acute
et}}(\Spec
(\mathcal{O}_{F,S}[\frac{1}{\ell}]),\Z_\ell(1-r))$ for the Chern class homomorphism constructed by
Soul\'e \cite{soule} and Dwyer and Friedlander \cite{df}. We also set ${\rm Tr}_{\C/\R} := 1+\tau \in \Z[G_{\C/\R}]$. 

\begin{lemma}\label{construction}
There exist finitely generated $G$-modules $X_r$ and $Y_r$ which possess all of the following properties.
\begin{itemize}
\item[(i)] $X_{r,{\rm tf}} \subseteq C_{r,{\rm tf}}$ and $Y_{r,{\rm tf}} = {\rm Tr}_{\C/\R}(B_r) \subseteq B_r^{G_{\C/\R}}$ and both $\Z'\otimes X_{r,{\rm tf}} = \Z'\otimes C_{r,{\rm tf}}$ and $\Z'\otimes Y_{r,{\rm tf}} = \Z'\otimes B_r^{G_{\C/\R}}$.

\item[(ii)] For each prime $\ell$ there is an isomorphism of $\Z_\ell[G]$-modules of the form
$\Z_\ell\otimes X_r \cong H^1_{{\rm \acute et}}(\Spec
(\mathcal{O}_{F,S}[\frac{1}{\ell}]),\Z_\ell(1-r)).$ If $\ell$ is odd
 the induced isomorphism $\Z_\ell\otimes X_{r,{\rm tf}} = \Z_\ell\otimes
K_{1-2r}(\mathcal{O}_{F,S})_{\rm tf} \rightarrow H^1_{{\rm \acute
et}}(\Spec (\mathcal{O}_{F,S}[\frac{1}{\ell}]),\Z_\ell(1-r))_{\rm
tf}$ coincides with that induced by ${\rm ch}^1_{\ell,1-r}$.

\item[(iii)] $X_{r,{\rm tor}} = \mu_{1-r}(F)$ and $Y_{r,{\rm tor}} = \bigoplus_{\ell}H^2_{{\rm \acute et}}(\Spec (\mathcal{O}_{F,S}[\frac{1}{\ell}]),\Z_\ell(1-r))$ where $\ell$ runs over all primes.
\item[(iv)] If the Quillen-Lichtenbaum Conjecture is valid for $F$ and $r$, then one has $\Z'\otimes X_{r} = \Z'\otimes
 C_r$ and $\Z'\otimes Y_{r,{\rm tor}} \cong \Z'\otimes K_{-2r}(\mathcal{O}_{F,S})$.
\end{itemize}
\end{lemma}

\begin{proof} All claims in Lemma~\ref{construction} except for the equality $Y_{r,{\rm tf}} = {\rm Tr}_{\C/\R}(B_r)$ follow directly from the
 constructions of \cite[\S11.1]{ltav} (where $X_r$ and $Y_r$ correspond to the modules $N'_{r,0}$ and $N'_{r,1}$ respectively). Since for every
 odd prime $p$ one has $\Z_p\otimes Y_{r,{\rm tf}} = \Z_p\otimes_{\Z'}(\Z' \otimes Y_{r,{\rm tf}}) =
 \Z_p\otimes_{\Z'}(\Z'\otimes B_r^{G_{\C/\R}}) = \Z_p\otimes B_r^{G_{\C/\R}} = \Z_p\otimes {\rm Tr}_{\C/\R}(B_r)$ it therefore suffices to
 prove that $\Z_2\otimes Y_{r,{\rm tf}} = \Z_2 \otimes {\rm Tr}_{\C/\R}(B_r)$. To show  this we recall the construction of $Y_r$
 uses an isomorphism of $\Z_2[G]$-modules $ \Z_2\otimes Y_r \cong H^1(C_{2,r}^\bullet)$ with $C_{2,r}^\bullet :=
 R\Hom_{\Z_2}(R\Gamma_{c,{{\rm \acute et}}}(\mathcal{O}_{F,S}[\frac{1}{2}],\Z_2(r)),\Z_2[-2])$ where the subscript `$c$' denotes cohomology with compact support, whilst from the long exact cohomology
 sequences associated to the commutative diagram of exact triangles given in \cite[(114)]{bufl96} (with $p =2$ and $r$ replaced by $1-r$)
 one obtains an exact commutative diagram of $\Z_2[G]$-modules of
 the form

\[\minCDarrowwidth1em\begin{CD} 0 @> >> H^1(C_{2,r}^\bullet)_{\rm tf} @> >> \bigoplus_{w\mid\infty}H^0_{{\rm \acute et}}(F_w,\Z_2(-r)) @> >>  H^3_{{\rm \acute et}}(\mathcal{O}_{F,S}[\frac{1}{2}],\Z_2(1-r))@> >>0 \\
@.@.  @V \theta VV @V \cong VV\\
@. @. \bigoplus_{w\mid\infty} H^{3}_{{\rm \acute et}}(F_w,\Z_2(1-r)) @= \bigoplus_{w\mid\infty} H^3_{{\rm \acute et}}(F_w,\Z_2(1-r))\\
@. @. @V VV \\
@. @. 0.
\end{CD}\]
Here $w$ runs over all archimedean places of $F$ and each term
$H^0_{{\rm \acute et}}(F_w,\Z_2(-r))$ arises  from the explicit description of the
complex $R\Gamma_\Delta(F_w,\Z_2(r))^*[-3]$ that is given in
\cite[p. 1391]{bufl96}. Now $H_{{\rm \acute et}}^{3}(F_w,\Z_2(1-r))$ is isomorphic to
$\Z_2/2\Z_2$ if $w$ is real and $r$ is even and vanishes in all
other cases. Thus, since the homomorphism $\theta$ in the above
diagram is known to respect the direct sum decompositions of its
source and target (see the proof of \cite[Lem. 18]{bufl96}), the
diagram induces an isomorphism
\[ H^1(C_{2,r}^\bullet)_{\rm tf}\cong \bigoplus_{w'}2H_{{\rm \acute et}}^0(F_{w'},\Z_2(-r))\oplus \bigoplus_wH_{{\rm \acute et}}^0(F_w,\Z_2(-r))\]
where $w'$, resp. $w$, runs over all real, resp. complex, places of $F$. Finally we note that each choice of a topological generator of $\Z_2(-r)$ gives a natural isomorphism of the last displayed direct sum module with $\Z_2\otimes {\rm Tr}_{\C/\R}(B_r)$. \end{proof}

The isomorphism $\Regr$ combines with the two equalities $Y_{r,{\rm tf}} = {\rm Tr}_{\C/\R}(B_r)$ and $\Q\otimes X_r = \Q\otimes C_r$ coming from Lemma~\ref{construction}(i) to imply the existence of homomorphisms of $G$-modules
\[ \lambda: Y_r\rightarrow X_r\]
that have both finite kernel and finite cokernel. In particular, for each such $\lambda$, the induced map $\Q\otimes\lambda: \Q\otimes B_r^{G_{\C/\R}} = \Q\otimes Y_r \to \Q\otimes X_r = \Q\otimes C_r$ is bijective and
so for each character $\psi = \chi^\gamma$, with $\gamma\in \Gamma$, we may set $R_r^\lambda (\psi ) := R_r^{\Q\otimes\lambda} (\psi )$. For each integer $j$ with $1\le j\le \chi(1)$ we then consider the composite
homomorphism of $\mathcal{O}$-modules
\[ t^j_{\lambda,\psi}: Y^j_{r,\psi}\xrightarrow{\lambda^j_\psi}
X^j_{r,\psi} \xrightarrow{t^j(X_r,\psi)} X^{j,\psi}_r.
\]
We write ${\rm Fit}_\mathcal{O}(M)$ for the Fitting ideal of a
finitely generated
 $\mathcal{O}$-module $M$ (so in particular ${\rm Fit}_{\mathcal{O}}(M) \subseteq \mathcal{O}$) and if $f: M \to M'$ is a homomorphism of finitely generated $\mathcal{O}$-modules
 that has both finite kernel and finite cokernel we
 define a fractional $\mathcal{O}$-ideal by setting $q(f) := {\rm
Fit}_\mathcal{O}(\cok(f)){\rm Fit}_\mathcal{O}(\ker(f))^{-1}.$

\begin{conj}\label{ssc}
(The `modified Lichtenbaum-Gross Conjecture')
Let $S$ be any finite set of places of $k$ that contains $S_\infty$ and all places that ramify in $F/k$. Then for every homomorphism $\lambda$ and index $j$ as above one has
\begin{equation}\label{SC}
  \frac{L_S^*(r,\check\chi^\alpha)}{R_r^\lambda (\chi^\alpha )}
=
  \frac{L_S^*( r,\check\chi)^\alpha}{R_r^\lambda(\chi)^\alpha}
\end{equation}
for all $\alpha \in {\rm Aut}(\C)$ and also
\begin{equation}\label{LGC}
 \frac{L_S^*(r,\check\chi^\gamma)}{R^\lambda_r(\chi^\gamma)} \mathcal{O} = q (t^j_{\lambda,\chi^\gamma})^{-1}
\end{equation}
for all $\gamma\in \Gamma$.
\end{conj}

\begin{remark}\label{ssc rem}\hfill

(i) After taking account of the isomorphism (\ref{translate}), the equality (\ref{SC}) coincides with the central conjecture of Gross in \cite{gross} (which is often referred to as the `Gross-Stark Conjecture'). The equality (\ref{LGC}) was first explicitly formulated by Chinburg, Kolster, Pappas and Snaith in \cite[Conj. 6.12]{ckps} (where it is referred to as a `modified Lichtenbaum-Gross Conjecture'). The notation of
loc.\ cit.\ is however different from that used here and the necessary translation is described in \cite[\S11.3]{ltav}.

(ii) The `equivariant Tamagawa number' $T\Omega(h^0(\Spec (F))(r),\Z[G])$ that is defined (unconditionally) by Flach and the first author in \cite[Conj. 4(iii)]{bufl00} is an element of the relative algebraic $K$-group $K_0(\Z[G],\R[G])$ and the relevant case of the equivariant Tamagawa number conjecture predicts that $T\Omega(h^0(\Spec (F))(r),\Z[G])$ vanishes. In \cite[\S11.1]{ltav} it is shown that Conjecture~\ref{ssc} is valid if and only if $T\Omega(h^0(\Spec (F))(r),\Z[G])$ belongs to the kernel of a natural homomorphism $\rho^\chi_*: K_0(\Z[G],\R[G]) \to K_0(\mathcal{O},\C)$. When combined with the good functorial properties of $T\Omega(h^0(\Spec (F))(r),\Z[G])$ under change of extension $F/k$ (that are proved in \cite[Prop. 4.1b)]{bufl00}), this observation implies that the validity of Conjecture~\ref{ssc} is unchanged if one replaces $F$ by any subfield $F'$ of $F$ that is Galois over $k$ and such that $\chi$ factors through the projection $G\to G_{F'/k}$.
\end{remark}

\subsection{}
We now reinterpret the conjectural equality (\ref{LGC}). To do this
we set $\psi := \chi^\gamma$ for some fixed $\gamma\in \Gamma$ and $n:= \psi(1) = \chi(1)$. We choose an integer $j$
with $1\le j\le n$ and set
 $\rho_r := \dim_E(V^j_{\psi}\otimes _{\Z [G]}B^{G_{\C/\R}}_r) = \dim_E(V^j_{\psi}\otimes _{\Z [G]}Y_r) =
\dim_E(E\otimes_\mathcal{O} Y^j_{r,\psi}).$ Then $\rho_r$ is independent of both $j$ and $\gamma$ and the argument of Lemma \ref{ord=mult} shows that the function $L_S(s,\check\psi)$ vanishes to order $\rho_r$ at $s=r$. Now (\ref{SC}) implies that the quotient $L_S^{(\rho_r)}(r,\check\psi)/R_r^\lambda (\psi)$ belongs to $E$ and, after unwinding the definition of $R_r^\lambda (\psi)$, this implies that there is an equality of $E$-spaces
\begin{equation*}
  L_S^{(\rho_r)}(r,\check\psi)\cdot{\wedge}^{\rho_r}_E(V^j_\psi\otimes_{\Z [G]}Y_r)
=
  \Regr^{(\psi)}({\wedge}^{\rho_r}_E(V^j_\psi\otimes_{\Z [G]} C_r))
,\end{equation*}
where $\Regr^{(\psi)}: \C\otimes_E{\wedge}^{\rho_r}_E(V^j_\psi\otimes_{\Z [G]} C_r) \xrightarrow{\sim}
\C\otimes_E {\wedge}^{\rho_r}_E(V^j_\psi\otimes_{\Z [G]} Y_r)$ is the isomorphism of $\C$-spaces induced in the obvious way by $\Regr$.
 In the next result we show that (\ref{LGC}) implies a natural integral refinement of this equality.


\begin{prop}\label{firstver}
Assume that (\ref{LGC}) is
valid for $\chi$. Set $K:= F^{\ker(\chi)}$ and $\psi := \chi^\gamma$ with $\gamma \in \Gamma$.
\begin{itemize}
\item[(i)] Then in $\C\otimes_E {\wedge}^{\rho_r}_E(V^j_\psi\otimes_{\Z [G]} Y_r) = \C\otimes_{\mathcal{O}}{\wedge}^{\rho_r}
_{\mathcal{O}}Y^j_{r,\psi}$ one has
\[ |G|^{\rho_r}L_S^{(\rho_r)}(r,\check\psi){\rm
Fit}_{\mathcal{O}}(\mu_{1-r}(K)^{j,\psi}){\wedge}^{\rho_r}
_{\mathcal{O}}(Y^j_{r,\psi})_{\rm tf}\!\! = \!{\rm
Fit}_{\mathcal{O}}((Y^j_{r,\psi})_{\rm tor})\Regr^{(\psi)}(
 {\wedge}^{\rho_r}_{\mathcal{O}}(X_r^{j,\psi})_{\rm tf}).\]
%
\item[(ii)] For each integer $a$ with $1 \le a\le \rho_r$ fix $\sigma_a$ in $\emb{F}$ and define $\tilde\sigma_a := {\rm
Tr}_{\C/\R}((2\pi i)^{-r}\sigma_a) \in
 Y_{r,{\rm tf}}$. Also set $\Sigma':= \{\sigma_a:1\le a\le \rho_r\} \subseteq \emb{F}$. Then for every
 $d'$ in ${\rm Fit}_\mathcal{O}(\mu_{1-r}(K)^{j,\psi})$ there exists a unique element $u^j_{\Sigma',\psi}(d')$ of $\,{\rm
Fit}_{\mathcal{O}}((Y^j_{r,\psi})_{\rm tor})
\wedge_\mathcal{O}^{\rho_r}(n^{-1}(C_r^{j,\psi})_{\rm tf})\,$ for
which one has 
\begin{equation*}
\Regr^{(\psi)}(u^j_{\Sigma',\psi}(d')) =
 d'(n^{-1}|G|)^{\rho_r}L_S^{(\rho_r)}(r,\check\psi){\wedge}_{a
=1}^{a =\rho_r} f^j_{\psi}(\tilde\sigma_a).\end{equation*}
\end{itemize}
\end{prop}

\begin{proof}
We set $\rho := \rho_r$. To derive the equality of claim (i) from the equality (\ref{LGC}) in Conjecture~\ref{ssc} we note that
\begin{align*}
   L_S^{(\rho_r)}(r,\check\psi)\cdot {\wedge}^{\rho}_{\mathcal{O}}(Y^j_{r,\psi})_{\rm tf}
&= R^\lambda_r(\psi)q (t^j_{\lambda,\psi})^{-1}{\wedge}^{\rho}_{\mathcal{O}}(Y^j_{r,\psi})_{\rm tf}
\\
&= q(t^j_{\lambda,\psi})^{-1}(\Regr^{(\psi)}\circ {\wedge}^\rho_\C (\C\otimes_\mathcal{O}\lambda^j_\psi))
   ({\wedge}^\rho_{\mathcal{O}}(Y^j_{r,\psi})_{\rm tf})
\\
&= \frac{q (\lambda^j_\psi)}{q(t^j_{\lambda,\psi})}
   \frac{{\rm Fit}_\mathcal{O}((Y^j_{r,\psi})_{\rm tor})}{{\rm Fit}_\mathcal{O}((X^j_{r,\psi})_{\rm tor})}
   \Regr^{(\psi)}({\wedge}^\rho_{\mathcal{O}}(X^j_{r,\psi})_{\rm tf})
\\
&= q (t^j(X_{r},\psi))^{-1}\frac{{\rm Fit}_\mathcal{O}((Y^j_{r,\psi})_{\rm tor})}{{\rm Fit}_\mathcal{O}((X^j_{r,\psi})_{\rm tor})}
   \Regr^{(\psi)}({\wedge}^\rho_{\mathcal{O}}(X^j_{r,\psi})_{\rm tf})
\\
&= |G|^{-\rho}\frac{{\rm Fit}_\mathcal{O}((Y^j_{r,\psi})_{\rm tor})}{{\rm Fit}_\mathcal{O}((X^{j,\psi}_r)_{\rm tor})}
    \Regr^{(\psi)}({\wedge}^\rho_{\mathcal{O}}(X_r^{j,\psi})_{\rm tf})
\,.
\end{align*}
The first equality here follows immediately from (\ref{LGC}), the second from the definition of $R^\lambda_r(\psi)$
and the fourth from the equality $q(t^j_{\lambda,\psi}) =
q(t^j(X_r,\psi))q(\lambda^j_\psi)$ which is a consequence of the
kernel-cokernel sequence of the composite $t^j_{\lambda,\psi} =
t^j(X_r,\psi)\circ \lambda^j_\psi$. In addition, the third and fifth
displayed equalities follow by applying Lemma~\ref{stolen} below with $f$ equal to $\lambda^j_\psi : Y^j_{r,\psi} \to
 X^j_{r,\psi}$ and $t^j(X_r,\psi): X^j_{r,\psi} \to X_r^{j,\psi}$
respectively.

To deduce the equality of claim (i) from the above displayed
formula it only remains to show that $(X^{j,\psi}_r)_{\rm
tor} = \mu_{1-r}(K)^{j,\psi}$. But $(X^{j,\psi}_r)_{\rm tor} := (T_\psi^j\otimes X_r)^G_{\rm tor} =
 (T_\psi^j\otimes X_{r,{\rm tor}})^G = (T_\psi^j\otimes X^{G_{F/K}}_{r,{\rm tor}})^G =: (X^{G_{F/K}}_{r,{\rm tor}})^{j,\psi}$ where the second  and third equalities follow from the fact that $T_\psi^j$ is torsion-free and, in the latter case, that $G_{F/K}=\ker(\psi)$ acts trivially on $T_\psi^j$. The required equality is therefore true because Lemma~\ref{construction}(iii) implies $X^{G_{F/K}}_{r,{\rm tor}} = \mu_{1-r}(F)^{G_{F/K}} := H^0(G_{\Q^c/F},\Q/\Z(1-r))^{G_{F/K}} = H^0(G_{\Q^c/K},\Q/\Z(1-r)) =:\mu_{1-r}(K)$. 

To prove claim (ii) we note that the element ${\wedge}_{a =1}^{a
=\rho} f^j_{\psi}(\tilde\sigma_a)$ belongs to
$\wedge_\mathcal{O}^{\rho}(Y^j_{r,\psi})_{\rm tf}$. The existence of
an element $u^j_{\Sigma',\psi}(d')$ of ${\rm
Fit}_{\mathcal{O}}((Y^j_{r,\psi})_{\rm
tor})\wedge_\mathcal{O}^{\rho}(n^{-1}(X_r^{j,\psi})_{\rm tf})$ which
satisfies the displayed equality in claim (ii) therefore follows
directly from claim (i). It thus suffices to note that Lemma
\ref{construction}(i) implies $(X_r^{j,\psi})_{\rm tf}\subseteq
(C_r^{j,\psi})_{\rm tf}$ and that the uniqueness of
$u^j_{\Sigma',\psi}(d')$ follows from the injectivity of ${\rm
Reg}_{r}^{(\psi)}$.\end{proof}

\begin{lemma}\label{stolen}
If $f: M \to N$ is any
homomorphism of finitely generated $\mathcal{O}$-modules that has
both finite kernel and finite cokernel, then there is an equality of $\mathcal{O}$-lattices
\[
  {\wedge}^d_E(E\otimes_{\mathcal{O}}f)({\wedge}^d_\mathcal{O}M_{\rm tf})
=
  q(f)\frac{{\rm Fit}_\mathcal{O}(M_{\rm tor})}{{\rm Fit}_\mathcal{O}(N_{\rm tor})}{\wedge}^d_\mathcal{O}N_{\rm tf}
\]
with $d := {\rm dim}_E(E\otimes_{\mathcal{O}} M)$.
\end{lemma}
\begin{proof} We consider the following exact commutative diagram

\begin{equation*}
\begin{CD}
@. \ker(f_{\rm tor}) @= \ker(f) @. 0\\
@. @V VV @V VV @V VV \\
0@ > >> M_{\rm tor} @> >> M @> >> M_{\rm tf} @> >> 0\\
@. @V{f_{\rm tor}}VV @V f VV @V f_{\rm tf}VV \\
0@ > >> N_{\rm tor} @> >> N @> >> N_{\rm tf} @> >> 0\\
@. @V VV @V VV @V VV \\
0@ > >> \cok(f_{\rm tor}) @> >> \cok(f) @> >> \cok(f_{\rm tf}) @> >> 0.\\
\end{CD}\end{equation*}
Here $f_{\rm tor}$ and $f_{\rm tf}$ denote the homomorphisms that
are induced by the given map $f$; the equality $\ker( f_{\rm tor}) =
\ker(f)$ and the injectivity of $f_{\rm tf}$ both follow from the
assumption that $\ker(f)$ is finite and the exactness of the bottom
row then follows from the Snake lemma. Now $M_{\rm tf}$, and hence also
$N_{\rm tf}$ since $\cok(f_{\rm tf})$ is finite, is a projective
$\mathcal{O}$-module of rank $d$ and so the definition of ${\rm
Fit}_\mathcal{O}(\cok(f_{\rm tf}))$ implies that
${\wedge}^d_E(E\otimes_{\mathcal{O}} f)({\wedge}^d_\mathcal{O}M_{\rm
tf}) = {\rm Fit}_\mathcal{O}(\cok(f_{\rm tf}))\cdot
{\wedge}^d_\mathcal{O}N_{\rm tf}.$ On the other hand, the exactness
of the bottom row and left hand column of the above diagram combines
with the multiplicativity of Fitting ideals on exact sequences of
finite $\mathcal{O}$-modules and the definition of $q(f)$ to imply
that ${\rm Fit}_\mathcal{O}(\cok(f_{\rm tf})) = {\rm
Fit}_\mathcal{O}(\cok(f)){\rm Fit}_\mathcal{O}(\cok(f_{\rm
tor}))^{-1} = q(f){\rm Fit}_\mathcal{O}(\ker(f)){\rm Fit}_\mathcal{O}(\cok(f_{\rm
tor}))^{-1} = q(f){\rm Fit}_\mathcal{O}(M_{\rm tor}){\rm
Fit}_\mathcal{O}(N_{\rm tor})^{-1}$. The claimed equality is thus
clear.
\end{proof}

\subsection{}\label{proof main}
We now prove Theorem~\ref{mainres}. To do this we set $n := \chi(1)$, $\psi := \chi^\gamma$ with $\gamma \in \Gamma$ and ${\rm pr}_\psi :=
n^{-1}|G|e_{\psi} = \sum_{g \in G}\check\psi(g)g \in \mathcal
{O}[G]$ and note that $\sum_{j = 1}^{j =
n}f^j_{\psi}$ is a decomposition of
$e_{\psi}$ as a sum of indecomposable idempotents in
$E[G]$. We set $K:= F^{\ker(\chi)} = F^{\ker(\psi)}$ and $d' := w_{1-r}(K)^n$. For the given embedding $\sigma$ we also set $\tilde\sigma :=
{\rm Tr}_{\C/\R}((2\pi i)^{-r}\sigma)$ (so $\tilde\sigma \in Y_{r,{\rm tf}}$ by Lemma~\ref{construction}(i)) and then for each index $j$ as above we define
\begin{equation}\label{epsilon-def}
\epsilon_\sigma(\psi,j) := \frac{d' |G|}{n}
\frac{L_S'(r,\check\psi)}{R^\lambda_r(\psi)}
 \lambda^j_\psi (f^j_{\psi}(\tilde\sigma)) \in \C\otimes_\mathcal{O}C_r^{j,\psi}
\,.
\end{equation}
%

We may assume that $\rho_r = 1$ since if $\rho_r > 1$, then
$L_S'(r,\check\chi^\delta) =0$ for all $\delta\in \Gamma$ and so
 Conjecture~\ref{chinlike} is obviously valid. Then, since $\rho_r = 1$ one has $L_S^*(r,\check\psi)
= L_S'(r,\check\psi)$ and so, as $f^j_{\psi}(\tilde\sigma)$ is a
non-zero element of the dimension one $E$-space
$E\otimes_\mathcal{O}Y^j_{r,\psi}$, the definition of
$R^\lambda_r(\psi)$ implies that
\[ \Regr^{(\psi)}(\epsilon_\sigma(\psi,j)) = d'n^{-1}|G|L_S'(r,\check\psi)\cdot
f^j_{\psi}(\tilde\sigma).\]
Now the finite group $\mu_{1-r}(K)$ is cyclic and so the (finite)
$\mathcal{O}$-module $\mu_{1-r}(K)^{j,\psi} \subseteq T^j_\psi\otimes
\mu_{1-r}(K)$ is both annihilated by $w_{1-r}(K)$ and generated by
(at most) ${\rm rk}_{\mathcal{O}}(T_\psi^j)=n$ elements. It follows in particular that $d'$
belongs to ${\rm Fit}_{\mathcal{O}}(\mu_{1-r}(K)^{j,\psi})$ and so the last
displayed formula implies that $\epsilon_\sigma(\psi,j)$ is equal to
 the element $u^j_{{\lbrace \sigma\rbrace},\psi}(d')$ of $n^{-1}{\rm Fit}_{\mathcal{O}}((Y^j_{r,\psi})_{\rm
tor}) (C^{j,\psi}_{r})_{\rm tf}$ that occurs in Proposition~\ref{firstver}(ii). Lemma~\ref{jacres}(ii) thus implies that 

\begin{equation}\label{last one added} \epsilon_\sigma(\psi,j)\in {\rm Fit}_{\mathcal{O}}((Y^j_{r,\psi})_{\rm tor})\otimes C_{r,{\rm tf}} \subseteq \mathcal{O}\otimes C_{r,{\rm tf}}.\end{equation}
The element $\overline{\epsilon}_\sigma(\psi,S) := \sum_{j = 1}^{j =
n}\epsilon_\sigma(\psi,j)$ therefore belongs to $\mathcal{O}\otimes C_{r,{\rm tf}}$ and satisfies the equality
\begin{equation}\label{barepsilon}
\begin{aligned}
& \mathrel{\phantom{=}} (2\pi i)^{r}\Regr(\overline{\epsilon}_\sigma(\psi,S))\\
& = (2\pi i)^{r}\sum_{j =1}^{j = n}d' n^{-1}|G|L_S'(r,\check\psi)\cdot f^j_{\psi}(\tilde\sigma)\\
& = (2\pi i)^{r}d' L_S'(r,\check\psi)\cdot{\rm pr}_\psi (\tilde\sigma)\\
& = d'\sum_{g \in G}L_S'(r,\check\psi)\check\psi(g)g(\sigma)+
    d'\sum_{g \in G}L_S'(r,\check\psi)(-1)^r\check\psi(g)g(\tau\circ\sigma)
\,.
\end{aligned}
\end{equation}

Now the assumed validity of (\ref{SC}) combines with (\ref{epsilon-def}) to imply that for all $\alpha\in {\rm Aut}(\C)$ and all $j$ one has
\begin{multline}\label{very last added} \epsilon_\sigma(\chi^\alpha,j)\! :=\! \frac{d' |G|}{n}
\frac{L_S'(r,\check\chi^\alpha)}{R^\lambda_r(\chi^\alpha)} \lambda^j_{\chi^\alpha} (f^j_{\chi^\alpha}(\tilde\sigma))\\
 = \left(\frac{d' |G|}{n}
 \frac{L_S'(r,\check\chi)}{R^\lambda_r(\chi)}
\lambda^j_\chi (f^j_{\chi}(\tilde\sigma))\right)^\alpha\!\! =:\!  \epsilon_\sigma(\chi,j)^\alpha \end{multline}
(where in the last two terms we use the natural semi-linear action of ${\rm Aut}(\C)$ on the space $\C\otimes_\mathcal{O}C_r^{j,\psi} \subset f_\psi^j\C[G]\otimes C_r$)  and hence also $\overline{\epsilon}_\sigma(\psi,S) =\overline{\epsilon}_\sigma(\chi,S)^\gamma$. Given this equality, it is now straightforward to check, by unwinding the definition of $\Regr$, that (\ref{barepsilon}) implies that for all $\sigma'$ in $\emb{F}$ the image of $\overline{\epsilon}_\sigma(\chi,S)^\gamma$ under $(2\pi i)^r{\rm reg}_{1-r,\sigma'}$ is equal to the right hand side of (\ref{conj eq}). To verify Conjecture~\ref{chinlike}, and hence complete the proof of Theorem~\ref{mainres}(i), it thus suffices to define $\epsilon_\sigma(\chi,S)$ to be any pre-image of $\overline{\epsilon}_\sigma(\chi,S)$ under the natural map $\mathcal{O}\otimes C_r \to \mathcal{O}\otimes C_{r,{\rm tf}}$.

The containment ${\rm Tr}_1(\epsilon_\sigma(\chi,S)) \in K_{1-2r}(F)$ in Theorem~\ref{mainres}(ii) follows
directly from claim~(i) and the result of Proposition~\ref{extraprop}(i)
with $d =1$. To prove the rest of claim (ii) we set $\overline{\epsilon}(\chi) :=
\overline{\epsilon}_\sigma(\chi,S)$. We also note that for each prime ideal $\mathfrak{p}$ of $\mathcal{O}$ the isomorphism class of the $\mathcal{O}_\mathfrak{p}$-module
$\mathcal{O}_\mathfrak{p}\otimes_{\mathcal{O}}Y_{r,\chi^\gamma}^j$ is independent of the choice of index $j$ and hence that ${\rm Fit}_{\mathcal{O}}((Y_{r,\chi^\gamma}^j)_{\rm tor}) = {\rm Fit}_{\mathcal{O}}((Y_{r,\chi^\gamma}^1)_{\rm tor})$ for all $j$. From the containment (\ref{last one added}) and equalities (\ref{very last added}) we therefore know that both 
$\overline{\epsilon}(\chi)^\gamma$ belongs to ${\rm Fit}_{\mathcal{O}}((Y_{r,\chi^\gamma}^1)_{\rm tor})\otimes C_{r,{\rm tf}}$
and that $\overline{\epsilon}(\chi^\gamma)= \overline{\epsilon}(\chi)^\gamma e_{\chi^\gamma}$
and so for every $\phi$ in $\Hom_G(C_r,\Z [G])\subseteq
\Hom_{E[G]}(E\otimes C_{r,{\rm
tf}},E[G])$ the element $\phi(\overline{\epsilon}(\chi)^\gamma)$ belongs to ${\rm Fit}_{\mathcal{O}}((Y_{r,\chi^\gamma}^1)_{\rm tor})\cdot\mathcal{O}[G]e_{\chi^\gamma}$. Now Lemma~\ref{construction}(i) implies that $\Z'\otimes
Y_{r,{\rm tf}}$ $= \Z'\otimes B_r^{G_{\C/\R}}$ is a projective $\Z'[G]$-module so that $\Z'\otimes
(Y^1_{r,\chi^\gamma})_{\rm tor}$ identifies with $\Z'\otimes
(Y_{r,{\rm tor}})^1_{\chi^\gamma}$ and hence that $\phi(\overline{\epsilon}(\chi)^\gamma)\in \Z'\otimes{\rm Fit}_{\mathcal{O}}((Y_{r,{\rm tor}})^1_{\chi^\gamma})\cdot\mathcal{O}[G]e_{\chi^\gamma}$.
 By applying \cite[Lem. 11.1.2(i)]{dals} we can therefore deduce that the element $\chi(1)^{-1}|G|^2\phi(\overline{\epsilon}(\chi)^\gamma)$ belongs to $\Z'\otimes {\rm Ann}_{\Z[G]}(Y_{r,{\rm tor}})$. The displayed containment in Theorem~\ref{mainres}(ii) thus follows directly from the equality
$\phi({\rm Tr}_1(\epsilon_\sigma(\chi,S))) = \sum_{\gamma \in \Gamma}\phi(\overline{\epsilon}(\chi)^\gamma)$ and the description of $Y_{r,{\rm tor}}$
given in Lemma~\ref{construction}(iii). This completes our proof of Theorem~\ref{mainres}.

\section{On $ K_3$ and the regulator}\label{k3 and the reg}

In preparation for describing some numerical evidence for Conjecture~\ref{chinlike}
we now make precise the relation between a version of
the (second) Bloch group and $ K_3 $ of a field, and, if the field is
a number field, the Beilinson regulator map. This result may itself be
of some independent interest and so, in order to keep open the
possibility of extending it to higher Bloch groups, we have used the
approach of~\cite{dJ95} rather than the potentially more precise
result of~\cite[Th.~5.2]{SusXkof}.

Let $ F $ be a field, and let $ K_3(F)^\ind $ be the quotient of $
K_3(F) $ by the image of the Milnor $K$-group $ K_3^M(F) $. We
recall that if $ F $ is a number field of signature $ [r_1,r_2] $
then $ K_3^M(F) \cong (\Z/2\Z)^{r_1} $ by~\cite[Th.~2.1(3)]{Ba-Ta} so
$K_3(F)^\ind_{\rm tf} = K_3(F)_{\rm tf}$ naturally, and this is a
free Abelian group of rank~$ r_2 $ by Remark \ref{useful remark}(iii). Finally, we
set
\begin{equation*}
\tilde\wedge^2 F^\times := F^\times \otimes F^\times / \langle (-x) \otimes x:  x \in
F^\times \rangle \,,
\end{equation*}
(this does not coincide with the usual exterior power $\wedge^2F^\times$
because of the negative sign in the denominator) and then write
$\delta_{2,F}$ for the homomorphism from the free Abelian group $
\Z[F^\times]$ on $ F^\times $ to $ \tilde\wedge^2 F^\times $ that is given by mapping a
generator $ \{x\} $ with $ x $ in $ F^\times $ to $ (1-x) \tilde\wedge\, x $ if $ x
\ne 1 $ and to $ 0 $ if $ x=1 $.

In the next proposition we refine various results in~\cite[\S2-\S5]{dJ95}.
Here we write $ D(z)$  for the
Bloch-Wigner dilogarithm $\C\setminus\{0\} \to \R(1)$ that is defined in \cite{bl00} by integrating the
function $ \log|w| \textup{d} i {\rm arg}(1-w) - \log|1-w| \textup{d} i {\rm arg}(w)
$ along any path from a point $ z_0\in \R\setminus \{0,1\} $ to $ z
$ (for $ z=1 $ one uses a limit).

\begin{thm}\label{Bprop}\hfill
\begin{itemize}
\item[(i)]
With notation as above, there is a homomorphism
\begin{equation*}
 \varphi_F : \ker(\delta_{2,F}) \to K_3(F)_{\rm tf}^\ind
\,,
\end{equation*}
which is natural up to sign.

\item[(ii)]
If $ F $ is a number field, then $ \varphi_F $ has finite cokernel.

\item[(iii)]
Moreover, there is a universal choice of sign, such that if $
F $ is any number field, and $ \sigma$ any embedding $F\to \C$, then
the composition
\begin{equation*}
 \ker(\delta_{2,F}) \buildrel{\varphi_F}\over{\to} K_3(F)_{\rm tf}^\ind
 = K_3(F)_{\rm tf} \buildrel{\sigma}\over{\to} K_3(\C)_{\rm tf}
\buildrel{\reg_2}\over{\to} \R(1)
\end{equation*}
is induced by mapping $ \{x\} $ to $ D(\sigma(x)) $.
\end{itemize}
\end{thm}

\begin{proof} At the outset we note that our set-up is essentially that of~\cite{bl00} with the various
modifications and improvements contained in~\cite{blo:ltd} and~\cite{dJ95} (but which are sometimes carried out only after
tensoring with $ \Q $). In particular, the reader may observe that
the construction of the element $ [x]_2' $ and the map to $
K_3(F)_{\rm tf}^\ind$ described below become those in~\cite[\S 3]{dJ95} after tensoring
with $ \Q $ and decomposing according to the Adams eigenspaces.

For any Noetherian regular ring $ R $ we let $ X_R = \P_R^1 \setminus \{t=1\} $, where
$ t $ is the standard affine coordinate on $ \P_R^1 $.  Then there is a long exact sequence of relative $ K $-groups
(terminating with $ K_0(\square) $),
\begin{equation*}
 \cdots \to K_n(X_R; \square) \to K_n(X_R ) \to K_n(\square) \to K_{n-1}(X_R; \square) \to \cdots
\end{equation*}
where $ \square $ consists of the subset of $ X_R $ where $
t=0,\infty $, i.e., two copies of $ \textup{Spec}(R) $. Since the
pullback $ K_n(R) \to K_n(X_R) $ along the natural map is an
isomorphism (by Quillen \cite[p.122]{qui73a}) and therefore the composition $
K_n(R) \to K_n(X_R) \to K_n(\square) \cong K_n(R)\oplus K_n(R)$ is
the diagonal map, we find an isomorphism
\begin{equation}\label{reliso}
 K_n(X_R;\square) \cong K_{n+1}(R)
\end{equation}
for $ n \ge0 $, with a map that is natural up to a universal sign.

We can combine localization with relativity under suitable assumptions (see~\cite[\S 2.2]{dJ95}).
In particular, if we take $  R = \Z[S,S^{-1}] $, then we have the exact sequence
(terminating with a map $ K_0(X_R; \square) \to 0 $)
\begin{equation*}
 \cdots \to K_2(X_R; \square) \to K_2(X_R \setminus \{t=S\}; \square)
\to K_1(R') \to K_1(X_R; \square) \to \cdots
\end{equation*}
with $ R' = \Z[ S,S^{-1},(1-S)^{-1} ] $.
Then $ K_1(R') \cong \langle -1, S, 1-S \rangle $, and
using~\cite[Lem.~3.14]{dJ95} one sees that the map $ K_1(R') \to K_1(X_R;\square) \cong K_2(R) $
maps $ (1-S) $ to $ \pm \{1-S,S\} = 0 $, so there
exists an element $ [S]_2^\sim $ in $K_2(X_R\setminus\{t=S\}; \square) $ with
image $ (1-S)^{-1} $ in $K_1(\Z[S,S^{-1}, (1-S)^{-1}]) $.  Because
$ K_2(X_R; \square) \cong K_3(R) \cong K_3(\Z) \oplus K_2(\Z) $ is torsion of
exponent~48~ (by Lee and Szczarba \cite{Le-Sz}) the element $ [S]_2^\sim $
is unique up to such torsion. We fix a choice of $[S]_2^\sim $ in what follows.

For the field $ F $, write $ F^\flat = F \setminus\{0,1\} $ as well
as $ X_{F,\loc} = X_F\setminus\{t=u \textup{ with } u \textup{ in }
 F^\flat \} $. By comparing the exact localization sequence
\begin{equation}\label{Iloc}
\cdots \to K_1(X_F;\square) \to K_1(X_{F,\loc};\square)  \to
\coprod_{u \in F^\flat} \Z \to K_0(X_F;\square) \to \cdots
\end{equation}
with the same one without relativity one sees as on~\cite[p.222]{dJ95} that
\begin{equation*}
 \I := K_1(X_{F,\loc} ; \square )
=
 \Bigl\{ f(t) = {\textstyle\prod}_{i=1}^n \left(\tfrac{t-a_i}{t-1}\right)^{n_i} \text{ with } f(0) = f(\infty) = 1 \Bigr\}
\end{equation*}
where $ n \ge 0 $, all $ n_i $ are in $ \Z $, and all $ a_i $ are in
$ F^\flat $. Tensoring the short exact sequence
\begin{equation*}
0 \to \I \to \coprod_{u \in F^\flat} \Z \to F^\times \to 0 \,,
\end{equation*}
which is part of~\eqref{Iloc}, with $ F^\times $ over $ \Z $, we obtain
the top row of the diagram

\begin{equation*}
\minCDarrowwidth1em
\begin{CD}
  0 @> >> \Tor_1^\Z(F^\times,F^\times) @> >> \I \otimes F^\times @> >> \coprod_{u \in F^\flat} F^\times @> >>  F^\times\otimes F^\times @> >> 0
\\
  @.  @V VV @V VV @\vert @V VV
\\
  0@> >>\frac{K_2(X_F;\square)}{\image(\coprod_{u\in F^\flat}K_2(F))} @> >> K_2(X_{F,\loc};\square)
                                @> d^\sim >>  \coprod_{u\in F^\flat} F^\times@> >> K_1(X_F;\square) @> >> 0.
\end{CD}
\end{equation*}

The bottom row is also obtained from~\eqref{Iloc}, using~\cite[Lem.~3.14]{dJ95} and the fact that $ K_2(F) $ is generated by
symbols $ \{a,b\} = a \cup b $. The second and last vertical maps
are induced by the cup product, thus giving rise to the first
vertical map which makes the diagram commutes.

Using~\cite[Lem.~3.14]{dJ95} yet again one sees that the image of $
\coprod_{u \in F^\flat} K_2(F) \to K_2(X_F; \square) $ under the
isomorphism $ K_2(X_F;\square) \cong K_3(F) $ in~\eqref{reliso}
equals the image of $ K_3^M(F) $ under the natural map to $ K_3(F)
$. So with $ K_3(F)^\ind = K_3(F)/\image(K_3^M(F)) $ and $ A $ the
image of $ \Tor_1^\Z(F^\times,F^\times) $ we obtain an exact sequence
\begin{equation*}
  0
\to
  K_3(F)^\ind/A
\to
  \frac{K_2(X_{F,\loc};\square)}{\image(\I\otimes F^\times)}
\buildrel{d}\over{\to}
  F^\times\otimes F^\times\to K_2(F)
\to
  0
\,,
\end{equation*}
which gives an isomorphism $ \ker(d) \cong K_3(F)^\ind/A $.

For $ x $ in $ F^\times $, let $ [x]_2^\sim $ be the image of $
x^*([S]_2^\sim) \in  K_2(X_F\setminus\{t=x\};\square) $ under the
localization $ K_2(X_F\setminus\{t=x\};\square) \to K_2( X_{F,\loc}
; \square ) $. Note that the boundary of $ [x]_2^\sim $ under $
d^\sim $ is $ (1-x)^{-1}_{|t=x} $  if $ x \ne 1 $ and is trivial if
$ x=1 $. Write $ [x]_2 $ for the image of $ [x]_2^\sim $ in $
K_2(X_{F,\loc};\square)/\image(\I\otimes F^\times)$, let $ B_2(F) $ be
the subgroup of $ K_2(X_{F,\loc};\square)/\image(\I\otimes F^\times)$
generated by the $ [x]_2 $ with $ x $ in $ F^\times $, and let $ d_2 $ be
the restriction of $ d $ to $ B_2(F) $.  We then have an inclusion
\begin{equation}\label{kerdmap}
 \ker(d_2) \to K_3(F)^\ind/A
\,.
\end{equation}

Now let $ \NN $ be the subgroup of $ B_2(F) $ generated by the
classes of $ [x]_2 + [1/x]_2 $ with $ x $ in $ F^\times $. Then we have a
commutative diagram
\begin{equation*}
\xymatrix{
\NN \ar[d] \ar[r]^-{d_{2|\NN}}   & d(\NN)      \ar[d]
\\
B_2(F) \ar[r]^-{d_2}    & F^\times \otimes F^\times \,. }
\end{equation*}
Next we note that $ d_2(\NN) = \langle (-x) \otimes x: x\in F^\times
\rangle \subseteq F^\times \otimes F^\times $ so $F^\times \otimes F^\times / d_2(\NN) =
\tilde{\wedge}^2F^\times$, and hence by taking quotients we get a short
exact sequence
\begin{equation}\label{sos}
0 \to \ker(d_{2|\NN}) \to \ker(d_2) \to \ker(d_2') \to 0
\end{equation}
where, setting $B'_2(F) := B_2(F)/\NN$, we write
\begin{equation*}
d_2' : B'_2(F) \to  \tilde{\wedge}^2F^\times
\end{equation*}
for the induced map. Writing $ [x]_2' $ for the image of $ [x]_2 $
in $ B'_2(F)$ we have that $ d_2'$ maps $[x]_2'$ to $ (1-x)^{-1} \tilde\wedge\,
x $ if $ x\ne1 $ and to $ 0 $ if $ x=1 $.

We know from~\cite[Lem.~3.7]{dJ95} that $ \ker(d_{2|\NN}) $ is torsion, so
combining~\eqref{sos} with the map~\eqref{kerdmap} and using
that $ A $ is torsion, we obtain a map
\begin{equation*}
\ker(d_2') \to K_3(F)^\ind_{\rm tf}
\end{equation*}
with torsion kernel.  We note that this map is independent of the
choice of $ [S]_2^\sim $, which was unique up to torsion. The map
$\varphi_F $ in Theorem~\ref{Bprop}(i) is now induced by mapping
$\{x\} $
  in $\ker(\delta_{2,F})\subseteq \Z[F^\times] $ to $ [x]_2' $ in
$\ker(d_2') \subseteq B'_2(F)$.

If $ F \subseteq \C $ then~\cite[Prop.~4.1]{dJ95} describes the
composition of the maps in $ \ker(d_2') \to K_3(F)_{\rm tf}^\ind \to
K_3(\C)_{\rm tf}^\ind $ with those in~\eqref{reg def} for $ r=-1 $,
namely
\begin{equation*}
K_3(\C)^\ind_{\rm tf} \to
 H_{\mathcal D}^1 (\textup{Spec}(\C), \R(2)) \cong \R(1)
\,.
\end{equation*}
That proposition states that the total map is induced by a map $
B_2'(F) \to \R(1) $ mapping $ [z]_2' $ to $ \pm D(z) $, with the
(universal) sign depending on the choice of the sign in the
isomorphism~\eqref{reliso}. (Note that in loc.\ cit.\ this was
expressed using the function $ P_2 (z) = D(z)/i $.) Claim~(iii) of
the proposition then follows because the construction of $ B_2'(F) $
and the map $ \varphi_F $ is natural, so that
\begin{equation*}
\xymatrix{ \ker(\delta_{2,F}) \ar[r] \ar[d] & B'_2(F) \ar[r] \ar[d]
& K_3(F)^\ind_{\rm tf}  \ar[d]
\\
\ker(\delta_{2,\C}) \ar[r] & B'_2(\C) \ar[r] & K_3(\C)^\ind_{\rm tf}
}
\end{equation*}
commutes, where the vertical maps are induced by the embedding
$ \sigma: F \to \C $. Finally, claim ~(ii) of the proposition follows from the
facts that, as $ F $ is a number field, after tensoring with $ \Q $,
$ \varphi_F $ induces an isomorphism (by \cite[Th.~5.3]{dJ95}) and $
K_3(F)_{\rm tf}$ is finitely generated (by Quillen \cite{qui73b}).
\end{proof}

\section{Numerical evidence}\label{ne}

In this section we provide corroborating numerical evidence for
Conjecture~\ref{chinlike} in the case that $r=-1$, $k = \Q$ and $S= S_\infty$
(so that Theorem~\ref{mainres} does not apply). For other interesting numerical work that is related to Conjecture \ref{chinlike} but 
uses elements in the $K$-group tensored with the rationals
see the articles of Besser, Buckingham, Roblot and the second author \cite[\S7]{BBdJR} and of Zagier and the third author \cite[\S5]{GZ}.

In the sequel we will use the symbol `$\approx$' to indicate a numerical identity that we have
checked to hold to many (and in all cases at least one hundred) decimal places.

We shall give examples of Galois extensions $F/\Q$ with
an embedding $\sigma:F\to\C$, and an irreducible complex character
$\chi$ of $G_{F/\Q}$ with $\chi(1)-\chi(\tau_\sigma)=2$,
so that $L(s,\check\chi)$ vanishes to order one at $s=-1$
by Remark~\ref{chin=2}, $ \chic{\chi}{\sigma,\sigma}=2 $ and $\embsp{-1,\chi}{F}=\emb{F}$.
Let $ \Q(\chi) \subset \C$ be the character field of $\chi$ with ring of integers $ \CO $.
In each case, we found elements $ \xi $ in $ \ker(\delta_{2,F}) $
and used the resulting $ \varphi_F(\xi) $ in $ K_3(F)_{\rm tf} $
to construct an element $ \beta $ of
$ \CO\otimes K_3(F)_{\rm tf} $ that satisfies $ e_{\sigma,\chi}\beta=\beta $,
as well as, for all $ \gamma $ in $ G_{\Q(\chi)/\Q} $,
\begin{equation*}
 (2\pi i)^{-1} \reg_{2,\sigma}(\beta^\gamma) \approx \gamma(e) L'(-1,\check\chi^\gamma)
\end{equation*}
for some $ e $ in $ E^\times $.
Using Remark~\ref{QabFremark} it is easy to check that $w_2(F^{{\rm ker}(\chi)})=24$
in all of our examples.  In each case we can write $ ef = \chic{\chi}{\sigma,\sigma} w_{-1}(\chi) = 2 \cdot 24^{\chi(1)}$
for some $ f $ in $ \CO $, so that $ \beta_\sigma(\chi,S_\infty) = f\beta $
satisfies the requirements of Proposition~\ref{equivalences}(ii),
thus verifying Conjecture~\ref{chinlike} for $ \chi $ and $ \sigma $.
In fact, in view of Lemma~\ref{chispecial}, this verifies Conjecture~\ref{chinlike} for all characters in $\{\chi^\gamma: \gamma \in \Gamma\}$ and for all $\sigma$ in $\emb{F}$.

We used Theorem~\ref{Bprop}(iii) and GP-PARI \cite{Pari} in order to compute
$\reg_{2,\sigma}(\beta^\gamma)$, and the latest version of MAGMA \cite{Magma} for the computation of
$L'(-1,\check\chi^\gamma)$ (but for our original experiments some of these values were provided by A.~Booker and
X.-F.~Roblot).
\\
The elements in $\ker(\delta_{2,F})$ were found, using
GP-PARI, roughly as follows. Let $T$ be a finite set of places of $\Q$ including the archimedean place. Then, in a first step, we produce
exceptional $T$-units $x_j\in\CO_{F,T}^\times$ (recall that $x_j$ is said to be `exceptional' if also $1-x_j\in \CO_{F,T}^\times$),
and we decompose both $x_j$ and $1-x_j$ with respect to a chosen set $\CB$ of
fundamental $T$-units (as provided by GP-PARI). The second step is then to look for linear relations among the elements
$(1-x_j)\tilde\wedge\, x_j$ of $\tilde\wedge{}^{\hskip -2pt 2}\langle\CB\rangle$, the latter
constituting a finite rank submodule of $\tilde\wedge{}^{\hskip -2pt 2}F^\times$.
(For more details on the underlying algorithm and implementation, we refer to work in preparation \cite{Kculator} and to the third author's personal web page.) 
A theorem of Bloch and Suslin, combined with the explicit formula for the rank of $K$-groups given in Remark \ref{useful remark}(iii), guarantees that, for sufficiently large
$T$, there are $\Z$-linear combinations $\xi$ of elements $\{x_j\}$ as above with $\xi \in \ker(\delta_{2,F})$
such that the $\varphi_F(\xi)$ generate a subgroup of finite index in $K_3(F)_{\rm tf}$
(cf.\ Theorem~\ref{Bprop}).
\\
For each natural number $m$ we now set $\zeta_m:=\exp(2\pi i/m)$
and $ \eta_m:=\zeta_m+\zeta_m^{-1} $.

\subsection{Dihedral representations}
Let $F$ be a Galois extension of $\Q$ with $G_{F/\Q}$ isomorphic to the dihedral group $D_n$
for an odd integer $n$.  If $F$ is not totally real, then the element $\tau_\sigma$ in
Conjecture~\ref{chinlike} has order $2$ and so for any irreducible $2$-dimensional character $\chi$
of $G$ one has $\chi(1)-\chi(\tau_\sigma) = 2$.

\subsubsection{}\label{dr}
Let $F=\Q(a)$ denote the totally complex
field of discriminant $-47^5$ that is the splitting field over $\Q$ of
the irreducible polynomial $x^{10} - 5x^9 + 12x^8 - 18x^7 + 20x^6 - 18x^5 + 10x^4 - x^3 - 2x^2 + x + 1$
in $\Q[x]$, of which $a$ is a root.
This field is such that $G=G_{F/\Q} \cong D_5$
is generated by an element $s$ of order~5 with
$s(a)= \frac15(3 a^9 - 12 a^8 + 25 a^7 - 34 a^6 + 38 a^5 - 33 a^4 + 16 a^3 - 7 a^2 + 2 a + 3)$
and an element $t$ of order~2 with $t(a)=1-a$.
The irreducible $2$-dimensional characters of $G_{F/\Q}$
take values in $ \Q(\eta_5)=\Q(\sqrt{5})$, and
form an orbit $\{\chi_1,\chi_2\}$ under the action of $\Gamma := G_{\Q(\sqrt5)/\Q}$, where
$\chi_j(s)=\zeta_5^j+\zeta_5^{-j}$ for $j \in \{1,2\}$.
MAGMA gives
\begin{equation*}
L'(-1,\check\chi_1) \approx-1.2094\ldots\,,\qquad L'(-1,\check\chi_2) \approx   -0.91109\ldots
\,.
\end{equation*}
The elements $u:=\frac15(-a^8 + 3 a^7 - 5 a^6 + 6 a^5 - 9 a^4 + 7 a^3 - 2 a^2 + 4 a + 3)$
and $u'=-{(1-u)^2}/{u}$
are exceptional units in $\CO_F$ and we found that $\xi:=2(\{u\}+\{u'\})$ lies in $\ker(\delta_{2,F})$.
We fix $\sigma$ with $\sigma(a)=1.367\ldots+0.197\ldots\cdot i$,
and let
$\beta=10e_{\sigma,\chi_1}\varphi_F(\xi)=\sum_{g\in G}\check\chi_1(g)(1-\tau_\sigma)g\varphi_F(\xi)$
in $\CO\otimes K_3(F)_{\rm tf}$, so that $e_{\sigma,\chi_1}\beta=\beta$.
We then find, for both $\gamma$ in $\Gamma$, that
\begin{equation}\label{D5eq}
(2\pi i)^{-1} \reg_{2,\sigma}(\beta^\gamma) \approx
\gamma(e) L'(-1,\check\chi_1^\gamma)
\end{equation}
for $e= \sqrt{5}$ in $E^\times$, so that $e$ does not divide
$\chic{\chi}{\sigma,\sigma}w_{-1}(\chi_1) = 2\cdot24^2$ in $\CO$.
We find a different $\beta$ as follows.
If $g$ in $G$ has $ \check\chi_1(g) \ne 0 $, then
$ (\check\chi_1(g)-2)/(\eta_5-2) $ is in $ \CO $, so we
can write $ 5e_{\chi_1}=(\eta_5-2)z+2z'$ with $z$ and $z'=\sum_{i=0}^{i=4}s^i$ in $ \CO[G] $.
If $\beta'$ is in $K_3(F)_{\rm tf}$, then $(1+t)z\beta'$ is in $\CO\otimes K_3(F)_{\rm tf}$
and $(\eta_5-2)(1+t)z\beta'=5e_{\chi_1} (1+t)\beta'$
because $(1+t)z'\beta'$ lies in $K_3(F)_{\rm tf}^G=\{0\}$.
Hence $\beta=(1-\tau_\sigma)(1+t)z\varphi_F(\xi)$ in $\CO\otimes K_3(F)_{\rm tf}$
also satisfies $e_{\sigma,\chi_1}\beta=\beta$.  For this choice of $\beta$
we find~\eqref{D5eq} is satisfied for both $\gamma$
with $e=-2-\sqrt{5}$,
so we may take $\beta_\sigma(\chi_1,S_\infty)=f\beta$ with
$f=(4-2\sqrt{5}) 24^2$ in~$\CO$.

\subsubsection{} Now let $F=\Q(a)$ denote the totally complex
field of discriminant $-71^7$ that is the splitting field over $\Q$ of
the irreducible polynomial
$x^{14} - 4x^{12} - x^{11} + 5x^{10} + 6x^9 + 16x^8 + 25x^7 + 16x^6 + 6x^5 + 5x^4 - x^3 -4x^2 + 1$
in $\Q[x]$, of which $a$ is a root.
This field is such that $G=G_{F/\Q}\cong D_7$ is generated by an element $s$ of order~7 with
$s(a)=\tfrac{1}{1139}
(1685a^{13}-343a^{12}-6831a^{11}-12a^{10}+8661a^9+7815a^8+25608a^7+36171a^6+17365a^5+6516a^4+6531a^3-4219a^2-5336a+1683)$
and an element $t$ of order~2 with $t(a)=a^{-1}$.
The irreducible $2$-dimensional characters of $G_{F/\Q}$
take values in $\Q(\eta_7)$, where $ \eta_7^3+\eta_7^2-2\eta_7-1=0 $,
and form an orbit $\{\chi_1,\chi_2,\chi_3\}$ under the action of $\Gamma := G_{\Q(\eta_7)/\Q}$, where
$\chi_j(s)=\zeta_7^j+\zeta_7^{-j}$ for $j \in \{1,2,3\}$.
MAGMA gives
\begin{equation*}
L'(-1,\check\chi_1) \approx -2.6049\ldots,\,\,\, L'(-1,\check\chi_2) \approx -2.1887\ldots,\,\,\, L'(-1,\check\chi_3) \approx -1.5689\ldots
\,.
\end{equation*}
Just as in \S\ref{dr} we found an element $\xi=\sum_{j\in J}n_j\{x_j\}$ in $\ker(\delta_{2,F})$
(with $|J|=14$, all $x_j$ exceptional units in $\CO_F$, one coefficient
$n_j$ equal to $-4$ and all thirteen others equal to ~$\pm2$). We fix $\sigma$ with $\sigma(a)= 0.450\ldots +0.163\ldots\cdot i$
and let $\beta=14e_{\sigma,\chi_1}\varphi_F(\xi)= \sum_{g\in G} \check\chi_1(g) (1-\tau_\sigma) g\varphi_F(\xi)$ in
$\CO\otimes K_3(F)_{\rm tf}$, so that $e_{\sigma,\chi_1}\beta=\beta$.
We then find, for all three $\gamma$ in $\Gamma$, that
\begin{equation}\label{D7eq}
(2\pi i)^{-1} \reg_{2,\sigma}(\beta^\gamma) \approx
\gamma(e) L'(-1,\check\chi_1^\gamma)
\end{equation}
for $e= 2\eta_7^2-\eta_7+1$ in $E^\times$, which has norm $ 49 $,
so that again $e$ does not divide
$\chic{\chi}{\sigma,\sigma}w_{-1}(\chi_1) = 2\cdot24^2$ in $\CO$.
 To find a more suitable element we write $ 7e_{\chi_1}=(\eta_7-2)z+2z'$ with $z$ and $z'=\sum_{i=0}^{i=6}s^i$ in $ \CO[G] $,
and one sees as before that
$\beta=(1-\tau_\sigma)(1+t)z\varphi_F(\xi)$ in $\CO\otimes K_3(F)_{\rm tf}$ is
fixed by $e_{\sigma,\chi_1}$.
For this $\beta$ we find that~\eqref{D7eq} is satisfied for all three $\gamma$
with $e=-3\eta_7^2-3\eta_7-1$, which has norm $ -13 $.
We may then take
\begin{equation*}
 \beta = \left( (-3\eta_7-3\eta_7^2) 14e_{\sigma,\chi_1} + (2-3\eta_7^2)  (1-\tau_\sigma)(1+t)z \right) \varphi_F(\xi)
\,,
\end{equation*}
which also satisfies $e_{\sigma,\chi_1}\beta=\beta$, and for
which~\eqref{D7eq} holds for all three $ \gamma $ with $ e=1 $,
and $\beta_\sigma(\chi_1,S_\infty) = 2\cdot24^2 \beta$.

\subsection{A tetrahedral representation}
Let $F$ denote the Galois closure of the field $F'=\Q(\theta)$ with $\theta^4=1-\theta$. Then $F$ has discriminant $283^{12}$ and
$G_{F/\Q}$ is isomorphic to the symmetric group $S_4$. In fact, since the polynomial $x^4+x-1$ has precisely two real roots, under the natural identification of $G_{F/\Q}$ with $S_4$ the element $\tau_\sigma$ corresponds to a transposition. For the (rational valued, $3$-dimensional) tetrahedral character $\chi_3$ of $G_{F/\Q}$ one therefore has
$\chi_3(1)-\chi_3(\tau_\sigma) = 2$.
In this case MAGMA gives $\,L'(-1,\check\chi_3) \approx 0.62475 \ldots\,\,$.
Note that the tetrahedral representation
occurs with multiplicity one inside $\Q\otimes K_3(F)$ by Lemma~\ref{ord=mult}
and that it has a unique one-dimensional $G_{F/F'}$-invariant subspace~$V$.
Because $K_3(F')_{\rm tf}$ has rank~1 and is of finite index in $K_3(F)_{\rm tf}^{G_{F/F'}}$,
both groups lie in~$V$ and $e_{\chi_3}$ acts on them as the identity.
This applies, in particular, to $\varphi_F(\{\theta\}) = \varphi_{F'}(\{\theta\})$.
If $\sigma$ is in $\emb{F}$ then we let $\beta=(1-\tau_\sigma)\varphi_F(\{\theta\})$
in $K_3(F)_{\rm tf}$, so that
$e_{\sigma,\chi_3}\beta=\beta$ and $\reg_{2,\sigma}(\beta)=2\reg_{2,\sigma}(\varphi_F(\{\theta\}))$.
For any $\sigma$ with $\sigma(\theta)=0.248\ldots+1.033\ldots\cdot i$ we find
\begin{equation*}
(2\pi i )^{-1} \reg_{2,\sigma}(\beta) \approx \frac12 L'(-1,\check\chi_3)
\,,
\end{equation*}
so we can take $\beta_\sigma(\chi,S_\infty) = 4 \cdot 24^3 \beta $.

\begin{remark}
In this case there is another rational valued, $3$-dimensional character $\chi_3'$ of $G_{F/\Q}$ that is obtained by multiplying $\chi_3$ by the alternating character. However, the function $L(s,\check\chi_3')$
vanishes to order $2$ at $s=-1$ and so Conjecture~\ref{chinlike} is trivially satisfied.
In the spirit of Proposition~\ref{firstver}, we investigated the value of the second derivative of $L(s,\check\chi_3')$ at $s=-1$. One has $\,L''(-1,\check\chi_3') \approx -10541.7335\ldots$. Further, with $F''$
denoting the fixed field of $F$ under some element of $G_{F/\Q}$ of order 4, we found two elements $\xi_1$, $\xi_2$ of
$\ker(\delta_{2,F''})$ (each being a linear combination of about one hundred terms), together with two embeddings
$\sigma_1$, $\sigma_2$ in $\emb{F}$, such that
\[
(2\pi i)^{-2} {\rm det}\bigl(\big(
{\reg}_{2,\sigma_a}(\varphi_F(\xi_b)\big)_{1\le a,b\le 2}\bigr)
 \approx \frac{1}{4} L''(-1,\check\chi_3')
\,.
\]
\end{remark}

\subsection{The Tate-Buhler-Chinburg representations}
As a final example we considered one of the fields $F$ of degree 48 over $\Q$ studied by Tate and Buhler~\cite{buhler}
and Chinburg~\cite[\S III.A.]{chin}. This field has discriminant $7^{24}19^{32}$ and $G:=G_{F/\Q}$ is isomorphic
to the amalgamated product
$ (\SL \times \Z/4\Z)/\langle ( ({\scriptstyle{\ol2\atop\ol0}{\ol0\atop\ol2}}), \ol2)\rangle $,
denoted $\amG$. This group has precisely six irreducible two-dimensional characters:
$ \chi_1, \check\chi_1 $ (corresponding to the representations
denoted $\sigma$ and $\Bar{\sigma}$ in~\cite{chin}, which each have character field $\Q(i)$)
and $ \chi_2,\check\chi_2,\chi_2', \check\chi_2' $
(corresponding to the representations $\rho,\Bar{\rho},\rho'$ and $\Bar{\rho}'$,
which each have character field $\Q(\zeta_{12})$).
Since for any $\sigma$ in $ \emb{F} $ the element $\tau_\sigma$ is
a non-central element of order~2, hence belongs to the conjugacy class
$A_0A_3$ in the character table \cite[Table I]{chin}, we see
that $\chi(1)-\chi(\tau_\sigma) = 2$ for every such character~$\chi$.

We now describe $ F $ and our identification of $ G $ with $ \amG $.
Note that $ P\GL \simeq S_4 $ by means of its action on $ \mathbb{P}_{\mathbb{F}_3}^1 $,
which also gives $ P\SL \simeq A_4 $.  As $ \SL $ has a unique
element of order~2 one sees easily that it is generated by any two non-commuting elements
$ \lambda_1 $ and $ \lambda_2 $ of order~3.
The disjoint conjugacy classes
of $ ({\scriptstyle{\ol1\atop\ol0}{\ol1\atop\ol1}}) $ and its
inverse contain all  eight elements of order~3,
so we can take the $ \lambda_i $ in the same conjugacy class.  Then
$ \amG $ is generated by $ (\lambda_1,\ol0) $, $ (\lambda_2,\ol0) $,
and either element of order~4 in its centre, and has at most~48 automorphisms.
 In fact all~48 automorphisms can be obtained by letting
$ P\GL $ act on $ \SL $ by conjugation and $ \{\pm 1\} $ act on $ \Z/4\Z $ by multiplication.
We may therefore identify $ G $ with $ \amG $ by specifying
two distinct conjugate elements $ g_1, g_2 $ of order~3 as well as a central element $ h $ of
order~4, and letting them correspond to
$ ( ({\scriptstyle{\ol1\atop\ol0}{\ol1\atop\ol1}}), \ol0) $,
$ ( ({\scriptstyle{\ol0\atop\ol2}{\ol1\atop\ol2}}), \ol0) $
and
$ ( ({\scriptstyle{\ol1\atop\ol0}{\ol0\atop\ol1}}), \ol1) $ respectively.

We note $ F $ is the splitting field of the
irreducible polynomial
$ f(x) = x^{16}-x^{15}+4x^{14}+x^{13}+2x^{12}+2x^{11}-9x^{10}+3x^9+19x^8+23x^7+13x^6+x^5+2x^4+3x^3+3x^2+3x+1 $
in $ \Q[x] $, so
$ F= \Q(a)(b) $ where $ f(a)=0 $ and $ b $ is a root of the factor
\begin{equation*}
 g(x) =
x^3 -(\tfrac{21430423}{9679918}a^{15}+\cdots+\tfrac{27880963}{4839959})x^2 +
\cdots+
(\tfrac{10585549}{4839959}a^{15} +\cdots + \tfrac{43460997}{9679918})
\end{equation*}
of $ f(x) $ in $ \Q(a)[x] $.
Let $ \sigma $ be the embedding of $F$ with
$ \sigma(a)= 1.254\ldots+0.583\ldots\cdot i $
and
$ \sigma(b)= 0.849\ldots-1.939\ldots\cdot i$, one of the roots of
$g^\sigma(x)=x^3-(0.222\ldots-2.544\ldots\cdot i)x^2-(1.658\ldots-1.108\ldots\cdot i)x-(0.827\cdots+0.252\ldots\cdot i)$.
We take distinct conjugate elements $ g_1, g_2 $ of order three
with $ g_1(a) = a $, $ \sigma(g_1(b)) = -0.612\ldots + 0.058\ldots\cdot i $,
$ \sigma(g_2(a)) = -0.612\ldots - 0.058\ldots\cdot i $,
$ \sigma(g_2(b)) =  1.254\ldots - 0.583\ldots\cdot i $,
and $ h $ the element of order four in the centre given by
$ \sigma(h(a)) = 0.586\ldots - 0.409\ldots\cdot i $
and $ \sigma(h(b)) =  -0.612\ldots - 0.058\ldots\cdot i $.
%
%
The field $ F $ is computationally difficult but fortunately we were able to find elements in
$ K_3(F)_{\rm tf} $ by searching in $ \ker(\delta_{2,F'})\subset\ker(\delta_{2,F}) $ for $ F'=\Q(c) $
where $ \sigma(c) =  1.472\ldots + 0.900\ldots\cdot i $.
(This field $ F' $ is one of six conjugate subfields
of degree~24 over $ \Q $, the only other subfield of this degree
being the fixed field of the normal subgroup $\{1,h^2\}$.)

The characters $ \chi_1 $ and $ \check\chi_1 $
form an orbit under the action of $G_{\Q(i)/\Q}$.  MAGMA gives
$ L'(-1,\check\chi_1) \approx -64.577\ldots + 631.991\ldots\cdot i $,
and its complex conjugate for $ L'(-1,\chi_1) $.
%
%
We found that, for both $ \gamma $ in $ G_{\Q(i)/\Q} $,
\begin{equation*}
 (2\pi i)^{-1} \reg_\sigma (\beta^\gamma) \approx 12 L'(-1, \check\chi_1^\gamma)
\end{equation*}
with $ \beta = \sum_{g\in G}\check\chi_1(g) (1-\tau_\sigma) g \varphi_F(\xi) $
for some $ \xi $ in $ \ker(\delta_{2,F'}) $,
so that $e_{\sigma,\chi_1}\beta=\beta$.
We may therefore take $\beta_\sigma(\chi_1,S_\infty)=f\beta$ with
$f=96$ in~$\CO=\Z[i]$.

The remaining characters $\chi_2,\check\chi_2,\chi_2'$ and $\check\chi_2'$ form an orbit under the action of
$G_{\Q(\zeta_{12})/\Q}$.  The corresponding $L$-values are
$L'(-1,\check\chi_2) \approx  -2.5823\ldots + 4.4538\ldots\cdot i $
and $L'(-1,\check\chi_2') \approx  -3.1252\ldots + 4.8866\ldots\cdot i$,
%
%
as well as their complex conjugates $L'(-1,\chi_2)$ and
$L'(-1,\chi_2')$. In this case we found that, for all $ \gamma $
in $ G_{\Q(\zeta_{12})/\Q} $,
\begin{equation*}
 (2\pi i)^{-1} \reg_\sigma (\beta^\gamma) \approx 12 L'(-1, \check\chi_2^\gamma)
\end{equation*}
with $ \beta = \sum_{g\in G}\check\chi_2(g) (1-\tau_\sigma) g \varphi_F(\xi) $ for
the same $ \xi $ in $ \ker(\delta_{2,F'}) $, so
that $e_{\sigma,\chi_2}\beta=\beta$.
We may therefore take $\beta_\sigma(\chi_2,S_\infty)=f\beta$ with
$f=96$ in~$\CO=\Z[\zeta_{12}]$.

\end{document}